\documentclass[a4paper, 11pt]{amsart}
\usepackage{mlmodern}

\usepackage[margin=2.5cm, headheight=13pt, headsep=13pt]{geometry}
\usepackage[colorlinks,citecolor=blue,filecolor=magenta,linkcolor=magenta,urlcolor=magenta]{hyperref}

\usepackage{amsfonts}
\usepackage{amssymb}
\usepackage{amsmath}
\usepackage{bbm}
\usepackage{mathrsfs,array}
\usepackage{tikz}
\usepackage{enumerate}
\usepackage{enumitem}

\usepackage[square,numbers]{natbib} % has a nice set of citation styles and commands
\def\T{\mathcal{T}}
\def\P{\mathcal{P}}
\def\S{\mathfrak{S}}

\def\des{\mathsf{des}}
\def\asc{\mathsf{asc}}

\def\r{\gamma}
\def\p{\pi}
\def\intt{\mathsf{int}}
 
\usepackage{xcolor}

\DeclareMathOperator{\DES}{DES}
\DeclareMathOperator{\codeg}{codeg}
\DeclareMathOperator{\argmax}{argmax}

\newcommand{\R}{\mathbb{R}}
\newcommand{\Z}{\mathbb{Z}}

\def\S{\mathfrak{S}}

\usepackage{array,tabularx}

\newtheorem{theorem}{Theorem}
\newtheorem{corollary}[theorem]{Corollary}
\newtheorem{lemma}[theorem]{Lemma}
\newtheorem{proposition}[theorem]{Proposition}

\theoremstyle{definition}
\newtheorem{rem}{Remark}
\newtheorem{question}[theorem]{Question}
\newtheorem{example}[theorem]{Example}
\newtheorem{definition}[theorem]{Definition}

\title{Colored Multiset Eulerian Polynomials}

\author{Danai Deligeorgaki}
\address{Department of Mathematics, KTH, Stockholm, Sweden}
\email{danaide@kth.se}

\author{Bin Han}
\address{Department of Obstetrics and Gynecology, Sahlgrenska Academy, University of Gothenburg,  41685 Gothenburg, Sweden}
\email{bin.han@gu.se}

\author{Liam Solus}
\address{Department of Mathematics, KTH, Stockholm, Sweden}
\email{solus@kth.se}

\date{\today}


\begin{document}

\begin{abstract}
Colored multiset Eulerian polynomials are a common generalization of MacMahon's multiset Eulerian polynomials and the colored Eulerian polynomials, both of which are known to satisfy well-studied distributional properties including real-rootedness, log-concavity and unimodality. 
The symmetric colored multiset Eulerian polynomials are characterized and used to prove sufficient conditions for a colored multiset Eulerian polynomial to be self-interlacing. 
The latter property implies the aforementioned distributional properties as well as others, including the alternatingly increasing property and bi-$\gamma$-positivity. 
To derive these results, multivariate generalizations of an identity due to MacMahon are deduced.  
The results are applied to a pair of questions, both previously studied in several special cases, that are seen to admit more general answers when framed in the context of colored multiset Eulerian polynomials. The first question pertains to $s$-Eulerian polynomials, and the second to interpretations of $\gamma$-coefficients. 
\end{abstract}

\maketitle

%%%%%%%%%%%%%%%%%%%%%%%%%%%%%%%%%%%%%
\section{Introduction}\label{sec1: intro}
%%%%%%%%%%%%%%%%%%%%%%%%%%%%%%%%%%%%%
For a positive integer $n$, we let $[n]:=\{1,2,...,n\}$ and $[n]_0:=\{0,1,...,n\}$. 
We denote by $M_{\textbf{m}}=\{1^{m_1},2^{m_2},...,n^{m_n}\}$, where $\textbf{m}:=(m_1,m_2,...,m_n)\in \Z_{>0}^n$, the multiset of cardinality $m:=m_1+\cdots+m_n$ containing $m_k$ copies of $k$, for $1\leq k \leq n$. 
The permutations of the multiset $M_{\textbf{m}}$, i.e., words $\pi=\pi_1\pi_2\ldots \pi_{m}$ in which $k$ appears exactly $m_k$ times, are called \emph{multiset permutations}, the set of which we denote $\S{_\textbf{m}}$. 
For $\textbf{r}:=(r_1,...,r_n)\in \Z_{>0}^n$, 
we denote by $\S{_\textbf{m}^{\textbf{r}}}$ the set of \emph{colored multiset permutations} of $M_\textbf{m}$, with color-vector $\textbf{r}:=(r_1,...,r_n)$; that is, words of the form $\pi^{\textbf{c}}=\pi_1^{c_1}\pi_2^{c_2}\ldots \pi_{m}^{c_{m}}$, where $k$ appears exactly $m_k$ times  and $1\leq c_k \leq r_{\pi_k}$,
for each $1\leq k \leq m$. 
When $\textbf{r}= \mathbf{1} := (1,...,1)$, the set $\S{_\textbf{m}^{\textbf{r}}}$ coincides with the multiset permutations $\S{_\textbf{m}}$. 
When $\mathbf{m} = \mathbf{1}$,  $\S{_\mathbf{1}^{\mathbf{1}}}$ is the permutations of $[n]$, $\S{_\mathbf{1}^{2\mathbf{1}}}$ is the {\em signed permutations} of $[n]$ and more generally $\S{_\mathbf{1}^{r\mathbf{1}}}$ for $r\geq 1$ is the \emph{$r$-colored permutations} of $[n]$, all of which appear frequently in algebraic combinatorics. 

Of particular interest are properties of descent statistics defined on  $\S{_\mathbf{1}^{r\mathbf{1}}}$ and $\S{_\mathbf{m}^{\mathbf{1}}}$, which admit a common generalization via a descent statistic on $\S{_\mathbf{m}^\mathbf{r}}$.
Let $k^{c_i}$ represent the multiset element $k$ with color $c_i\in[r_k]$.
We impose the \emph{color ordering} \cite{STEINGRIMSSON1994187} 
\[
1^1<2^1<\cdots<n^1<(n+1)^1<1^2<2^2<\cdots<n^2<1^3<\cdots<n^{\text{max}\{r_1,...,r_n\}}, 
\]
on the elements in the ground set $M_\mathbf{m}^{\mathbf{r}}$ of $\S{_\mathbf{m}^\mathbf{r}}$,
where we include the extra term $(n+1)^1$ by setting $\pi_{m+1}^{c_{m+1}}:=(n+1)^1$ in every colored permutation $\pi_1^{c_1}\pi_2^{c_2}\cdots \pi_{m}^{c_{m}}\in \S{_\textbf{m}^{\textbf{r}}}$. 
We say that $j\in [m]$ is a \emph{descent} (respectively, \emph{ascent}) of $\pi^{\textbf{c}}=\pi_1^{c_1}\pi_2^{c_2}\cdots \pi_{m}^{c_{m}}\pi_{m+1}^{c_{m+1}}$ if $\pi_j^{c_j}>\pi_{j+1}^{c_{j+1}}$ (respectively, $\pi_j^{c_j}<\pi_{j+1}^{c_{j+1}}$) according to the color ordering. 
We denote by $\DES(\pi^{\textbf{c}})$ the set of descents of a colored multiset permutation $\pi^{\textbf{c}}$, and we let $\des(\pi^{\textbf{c}}) = |\DES(\pi^{\textbf{c}})|$.
We also define $\text{ASC}(\pi^{\textbf{c}})$ and $\text{asc}(\pi^{\textbf{c}})$ analogously for ascents.

Early in the 20th century, MacMahon proved the following identity for the multiset permutations $\S{_{\mathbf{m}}^{\mathbf{1}}}$,

\begin{equation}
\label{macmahon's formula}\frac{\sum\limits_{\pi\in  \S{_\textbf{m}^\mathbf{1}}}x^{\text{des}(\pi)}}{(1-x)^{m+1}}= \sum_{t\geq 0}
 {\binom{t+m_1}{ m_1}}{\binom{t+ m_2}{ m_2}}\cdots 
 {\binom{t+ m_{n}}{ m_{n}}}x^t.
 \end{equation}
(see, for instance, \cite[Volume 2, Chapter IV, p. 211]{Macm}.)
When $\textbf{m} = \mathbf{1}$, MacMahon's formula recovers a well-known identity for the {\em $n$-th Eulerian polynomial}
\[
A_n = \sum_{\pi\in\S_n}x^{\des(\pi)},
\]
which is ubiquitous in enumerative and algebraic combinatorics.  
The Eulerian polynomials $A_n$ enjoy a variety of sought-after distributional properties including symmetry, unimodality, log-concavity, real-rootedness as well as $\gamma$-positivity \cite{PB}. 
When the Eulerian polynomials arise as special cases of larger families of generating polynomials, such as the \emph{$P$-Eulerian polynomials} \cite[Chapter 3]{stanley2011enumerative} or the \emph{$s$-Eulerian polynomials} \cite{savage2016mathematics}, many of these desirable distributional properties are often seen to extend to the larger family. 
In this paper, we consider the {\em colored multiset Eulerian polynomial} 
\[
A{{_\textbf{m}^{\textbf{r}}}} = \sum_{\pi^{\mathbf{c}} \in \S{_\textbf{m}^{\textbf{r}}}}x^{\des(\pi)}.
\]
These polynomials are a common generalization of both the \emph{$r$-colored Eulerian polynomials} $A{_{\mathbf{1}}^{r\mathbf{1}}}$ and MacMahon's \emph{multiset Eulerian polynomials} $A{_{\textbf{m}}^{\mathbf{1}}}$. 
The distributional properties of both $r$-colored Eulerian polynomials and MacMahon's multiset Eulerian polynomials have been studied extensively. 
For instance, the colored Eulerian polynomials were shown to be real-rooted, log-concave and unimodal \cite{savagevisontai}, and they are symmetric if and only if $r \in\{1, 2\}$.  
More recently, they were shown to possess a strong distributional property investigated by Br\"and\'en and Solus \cite{PS}; namely, they are interlaced by their own reciprocal.  
This property implies additional distributional properties including { bi-$\gamma$-positivity} and real-rootedness, unimodality and log-concavity of their {symmetric decomposition}.  

For MacMahon's multiset Eulerian polynomials, Simion \cite{simion} showed that $A{_{\mathbf{m}}^{\mathbf{1}}}$ is real-rooted and thus log-concave. 
Carlitz and Hoggatt \cite{carlitz1978generalized} further showed that $A{_{p\mathbf{1}}^{\mathbf{1}}}$ is symmetric for all $p \geq 1$, from which $\gamma$-positivity follows. 
Recently, Lin, Xu and Zhao \cite{multiseteulerian} gave a combinatorial interpretation of these $\gamma$-coefficients.

The colored Eulerian polynomials as well as the (uncolored) multiset Eulerian polynomial for $\textbf{m} = 2\mathbf{1}$ were further shown to be {$s$-Eulerian polynomials} by Savage and Visontai \cite{savagevisontai}.  
Savage and Visontai conjectured that the same is true for the Type-B analog $A{_{2\mathbf{1}}^{2\mathbf{1}}}$ \cite[Conjecture 3.25]{savagevisontai} in an effort to more fully describe the combinatorial generating polynomials that are \emph{$s$-Eulerian}; i.e., equal to some $s$-Eulerian polynomial. 
To settle this conjecture, Lin \cite[Theorem 6]{LZC} generalized MacMahon's formula to signed multiset permutations, showing that
\begin{align}
 \frac{\sum\limits_{\pi^{\mathbf{c}}\in  \S{_\textbf{m}^{2\mathbf{1}}}}x^{\text{des}(\pi)}}{(1-x)^{m+1}} & = \sum_{t\geq 0}\prod_{j=1}^{n} \bigg(\sum_{{{i=0}}}^{ m_j }
 {\binom{t+m_j-i}{ m_j-i}} 
 {\binom{t+ i-1}{ i}}
 \bigg)x^t, \label{lin's formula}\\
  \frac{\sum\limits_{\pi^{\mathbf{c}}\in  \S{_\textbf{m}^{2\mathbf{1}}}}x^{\text{des}(\pi)}}{(1-x)^{m+1}} &= \sum_{t\geq 0}\prod_{j=1}^{n}
 {\binom{2t+m_j}{ m_j}}
 x^t. \label{eqn:lin2}
% \end{split}
\end{align}

In this paper, we study how these observed properties of the colored Eulerian polynomials and MacMahon's multiset Eulerian polynomials extend to their common generalization of colored multiset Eulerian polynomials, providing a unified theory for the above results. 
In Section~\ref{sec:distributionalproperties}, we give a characterization of the symmetric (e.g. palindromic) $A{_{\textbf{m}}^{\textbf{r}}}$, generalizing the result of Carlitz and Hoggatt \cite{carlitz1978generalized}. We then apply it to show that all $A{_{\textbf{m}}^{\textbf{r}}}$ of degree at least $m$ are interlaced by their own reciprocal.  
This generalizes the result of \cite{PS} for colored Eulerian polynomials, and implies that these polynomials satisfy an extensive catalogue of distributional properties, including the recently popular alternatingly increasing property and bi-$\gamma$-positivity. 
To prove these results, an identity of multivariate generating functions, generalizing MacMahon's formula, is utilized. 
Hence, in Section~\ref{sec:generatingfunctions}, we first study generating functions associated to descents of colored multiset permutations, providing generalizations of identities for permutations, signed permutations  {and permutations with fixed number of colors} due to MacMahon,  {Steingr{\'{i}}msson}, Carlitz, Bagno and Biaglioli, as well as Beck and Braun. 

In Section~\ref{sec:interpretations}, we apply the results of Sections~\ref{sec:generatingfunctions} and~\ref{sec:distributionalproperties} to a pair of questions that are actively studied in several special cases. 
These questions may simply, and more generally, be framed as questions regarding colored multiset Eulerian polynomials:  
\begin{enumerate}
    \item (Question~\ref{quest:s-Eulerian}) When is a colored multiset Eulerian polynomial $s$-Eulerian?
    \item (Question~\ref{prob: gamma coeffs}) What is a combinatorial interpretation of the $\gamma$-coefficients of a bi-$\gamma$-positive colored multiset Eulerian polynomial?
\end{enumerate}
The first question provides a general framework for the investigations initiated by Savage and Visontai \cite{savagevisontai}. 
In this context, we prove a common generalization of the existing results on the colored Eulerian polynomials $A{_\mathbf{1}^{r\mathbf{1}}}$ and the polynomials $A{_{2\mathbf{1}}^{\mathbf{1}}}$ and $A{_{2\mathbf{1}}^{2\mathbf{1}}}$ (described above). 
We also observe that there exist (un)colored multiset Eulerian polynomials that are not $s$-Eulerian. 
Regarding the second question, we extend the results of Lin, Xu and Zhao \cite{multiseteulerian} on interpretations of $\gamma$-coefficients for multiset Eulerian polynomials.
In doing so we provide some basic results that align their investigations with a line of research proposed by Athanasiadis (Remark~\ref{rmk:gammacoeffs}).  
Perhaps worthy of note is the observation that the values of $\mathbf{m}$ and $\mathbf{r}$ for which Question~\ref{quest:s-Eulerian} and, respectively,  Question~\ref{prob: gamma coeffs} currently have a solution are closely related (see Proposition~\ref{prop: s-lecture hall} and Remark~\ref{rmk:gammacoeffs}).

\section{Generating function identities}
\label{sec:generatingfunctions}

In this section, we derive several identities of multivariate generating functions associated to descents (and ascents) in colored multiset permutations, generalizing some classic results. 
Throughout this section, we fix a positive integer $n$ and the positive integral vectors $\mathbf{m}=(m_1,...,m_n)$ and $\mathbf{r}=(r_1,...,r_n)$. 
We let $\pi^{\textbf{c}}=\pi_1^{c_1}\cdots \pi_m^{c_m} \pi_{m+1}^{c_{m+1}}$ denote a permutation in $\mathfrak{S}{_\textbf{m}^{\textbf{r}}}$.

\subsection{A multivariate identity}
\label{subsec:identity1}
To derive the desired results on the distributional properties of the polynomial $A{{_\textbf{m}^{\textbf{r}}}}$ in Section~\ref{sec:distributionalproperties}, we first give a generalization of MacMahon's formula~\eqref{macmahon's formula} to the colored multiset permutations.  
The required univariate identity can be recovered from a multivariate identity that also specializes to other previously observed generalizations and q-analogues of MacMahon's formula. 
The relevant multivariate identity is the following.

\begin{theorem}\label{thm: product formula multivariate}
 For positive integral vectors $\mathbf{m}=(m_1,...,m_n)$ and $\mathbf{r}=(r_1,...,r_n)$, we have
 
\begin{align}\label{first genera} & 
\sum_{\pi^{\textbf{c}} \in \S{_\textbf{m}^{\mathbf{r}}}}\frac{
z_{\pi_1}^{c_{1}-1}\cdots z_{\pi_m}^{c_{m}-1}
\prod\limits_{\substack{i\in[m+1]\setminus\{1\} \\ i-1 \in \text{DES}(\pi^{\textbf{c}})}}z_{\pi_i}^{r_{\pi_i}}\cdots z_{\pi_m}^{r_{\pi_m}}z_{m+1}}{\prod\limits_{i\in [m+1]}\big(1-z_{\pi_i}^{r_{\pi_i}}\cdots z_{\pi_m}^{r_{\pi_m}}z_{m+1}\big)}
& =\sum_{t\geq 0} z_{m+1}^t   
\prod_{k\in[n]} 
\begin{bmatrix}r_kt +m_k \\  m_k \end{bmatrix}_{z_{k}}.
\end{align}
\end{theorem}

The identity in Theorem~\ref{thm: product formula multivariate} specializes to several known identities. 
Setting $z_{m+1}=x$, $z_k=1$ and $\textbf{r} = \mathbf{1}$ recovers~\eqref{macmahon's formula}. 
Similarly,~\eqref{eqn:lin2} %the equality of the first and last terms in \eqref{lin's formula} 
is recovered by setting $z_{m+1}=x$, $z_k=1$ and $\textbf{r} = 2\mathbf{1}$.  Equations~\eqref{macmahon's formula} and~\eqref{eqn:lin2} also have $q$-analogs that are analogously recovered by an appropriate specialization.
When $\textbf{m} = \mathbf{1}$ and $\mathbf{r} = r\mathbf{1}$ for a positive integer $r$, \eqref{first genera} specializes to a formula of Steingr{\'{i}}msson \cite[Theorem 17]{STEINGRIMSSON1994187} for $z_{m+1}=x$ and $z_k=1$ (and its $q$-analog, \cite[Proposition 8.1]{Biagioli2009EnumeratingWP},                                                                                                        for $z_k=q$).  Theorem \ref{thm: product formula multivariate} can also be viewed as a generalization of the $q$-binomial theorem \cite[Theorem 2.1]{doi:10.1137/1021088} (case $n=1$).
Another interesting evaluation arises for $z_{m+1}=x$ and $z_k=q$ for all $k\in [n]$, for $\mathbf{r} = r\mathbf{1}$, where (a flag version of) a permutation statistic similar to the inversion-sequence statistic dmaj considered in \cite{SAVAGE2012850} appears on the enumerator in \eqref{first genera}. We define this statistic and present the identity in detail below. 

The proof of Theorem~\ref{thm: product formula multivariate} is combinatorial, and uses the notion of barred permutations popularized by Gessel and Stanley \cite{Gessel} and further  {developed} by Lin \cite{LZC} in the study of signed multiset permutations. 
A \textit{barred permutation} on $\pi^{\mathbf{c}} =\pi_1^{c_1}\cdots \pi_{m}^{c_{m}}(n+1)^1\in\S{_\textbf{m}^{\textbf{r}}}$
 is obtained by inserting one or more vertical bars
between letters in the word $\pi^{\mathbf{c}}$ such that there is at least one bar in every descent space of $\pi^{\mathbf{c}}$. 
For example, for the colored permutation $\pi^{\mathbf{c}}=1^22^12^22^21^13^1$ of the multiset $M_\textbf{m}^{\mathbf{r}}$ with $\mathbf{m}=(2,3)$ and $\mathbf{r}=(2,2)$ with descents in positions $1$ and $4$, $||1^2|2^12^2|2^2||1^13^1$ is a barred permutation on $\pi^{\mathbf{c}}$ but $||1^22^12^2|2^2||1^13^1$ is not.
Note that the bars can also be inserted at the $0$-th space before $\pi_1^{c_1}$, but not after the entry $\pi_{m+1}^{c_{m+1}}=(n+1)^1$; for instance, $1^2|2^12^22^2|1^13^1|$ cannot occur.
We denote by $B(\S{_\textbf{m}^{\mathbf{r}}})$ the set of barred permutations on $\S{_\textbf{m}^{\mathbf{r}}}$.
\begin{proof}[Proof of Theorem~\ref{thm: product formula multivariate}]
For a barred permutation $\sigma \in B(\S{_\textbf{m}^{\mathbf{r}}})$ on $\pi^{\textbf{c}} \in \S{_\textbf{m}^{\mathbf{r}}}$, let $d_i$ denote the number of bars between $\pi_{i-1}^{c_{i-1}}$ and $\pi_i^{c_i}$, for $i=2,...,m+1$, and let $d_{1}$ correspond to the number of bars to the left of $\pi_1^{c_1}$. We define the weight of $\sigma$ as

\[wt(\sigma):=z_{\pi_1}^{c_{1}-1}\cdots z_{\pi_m}^{c_{m}-1}
\prod\limits_{\substack{i\in[m+1] }}\big(z_{\pi_i}^{r_{\pi_i}}\cdots z_{\pi_m}^{r_{\pi_m}}z_{m+1}\big)^{d_{i}}
.\]

For example, for the barred permutation $\sigma= ||1^2|2^12^2|2^2||1^13^1$ on $\pi^{\textbf{c}} \in  \S{_\textbf{m}^{\mathbf{r}}}$ for $\textbf{m}=(2,3)$ and any vector of positive integers $\mathbf{r}=(r_1,r_2)$,
we compute \[wt(\sigma)=\ (z_{1}z_{2}^{2})(z_{1}^{2r_1}z_{2}^{3r_2}z_6)^2(z_{1}^{r_1}z_{2}^{3r_2}z_6)(z_{1}^{r_1}z_{2}^{r_2}z_6)(z_{1}^{r_1}z_6)^2\ =\ z_{1}^{8r_1+1}z_{2}^{10r_2+2}z_6^6.\] 
To prove the desired equality, we sum over the weights of all barred permutations in $B(\S{_\textbf{m}^{\mathbf{r}}})$ in two ways.

First, we fix a permutation $\pi^{\textbf{c}} \in \S{_\textbf{m}^{\mathbf{r}}}$ and sum over the set of barred permutations on $\pi^{\textbf{c}}$, which 
we denote by $B(\pi^{\textbf{c}})$.
The key observation here {(due to Gessel and Stanley \cite{Gessel})} %(as noted in \cite{LZC})
is that a barred permutation on $\pi^{\textbf{c}}$ can be obtained by inserting one bar
in each descent space and then inserting any number of bars in each space ($m+1$ spaces in total). 
Therefore, summing over the weights of all possible barred permutations on $\pi^{\textbf{c}}$ yields
\begin{align*}
& \sum_{\sigma \in B(\pi^{\textbf{c}})}wt(\sigma)
   \\ & = 
   z_{\pi_1}^{c_{1}-1}\cdots z_{\pi_m}^{c_{m}-1}
\prod\limits_{\substack{ i\in[m+1]\setminus\{1\}  \\  i-1 \in \text{DES}(\pi^{\textbf{c}})}}z_{\pi_i}^{r_{\pi_i}}\cdots z_{\pi_m}^{r_{\pi_m}}z_{m+1}
 \prod\limits_{\substack{ i\in [m+1]}}\Big(\sum\limits_{j\geq 0}\big(z_{\pi_i}^{r_{\pi_i}}\cdots z_{\pi_m}^{r_{\pi_m}}z_{m+1}\big)^j\Big) 
\\ & =\frac{
z_{\pi_1}^{c_{1}-1}\cdots z_{\pi_m}^{c_{m}-1}
\prod\limits_{\substack{i\in[m+1]\setminus\{1\} \\ i-1 \in \text{DES}(\pi^{\textbf{c}})}}z_{\pi_i}^{r_{\pi_i}}\cdots z_{\pi_m}^{r_{\pi_m}}z_{m+1}}{\prod\limits_{i\in [m+1]}\big(1-z_{\pi_i}^{r_{\pi_i}}\cdots z_{\pi_m}^{r_{\pi_m}}z_{m+1}\big)}.
\end{align*}
{Therefore, summing over all permutations $\pi^c\in \S{_\textbf{m}^{\mathbf{r}}}$  gives 
\begin{align*}
& \sum_{\sigma \in B(\S{_\textbf{m}^{\mathbf{r}}})}wt(\sigma)= \sum_{\pi^{\textbf{c}} \in \S{_\textbf{m}^{\mathbf{r}}}}\frac{
z_{\pi_1}^{c_{1}-1}\cdots z_{\pi_m}^{c_{m}-1}
\prod\limits_{\substack{i\in[m+1]\setminus\{1\} \\ i-1 \in \text{DES}(\pi^{\textbf{c}})}}z_{\pi_i}^{r_{\pi_i}}\cdots z_{\pi_m}^{r_{\pi_m}}z_{m+1}}{\prod\limits_{i\in [m+1]}\big(1-z_{\pi_i}^{r_{\pi_i}}\cdots z_{\pi_m}^{r_{\pi_m}}z_{m+1}\big)}.
\end{align*}}

Let us now partition $B(\S{_\textbf{m}^{\mathbf{r}}})$ into sets of barred permutations with exactly $t$ bars, denoted by $B_t(\S{_\textbf{m}^{\mathbf{r}}})$. 
Each $t$-barred permutation $\sigma \in B_t(\S{_\textbf{m}^{\mathbf{r}}})$ can be constructed by first placing $t$ bars in a line and then inserting any choice of $m_k$ (not necessarily distinct) colored copies of $k$ from $\{k^1,k^2,...,k^{r_k}\}$ into the $t+1$ spaces between adjacent bars, including the spaces on the left and right side (for $k=1,...,n$).
However, we add the condition that only ${1^1,2^1,...,n^1}$ can be inserted on the right side since we have attached $\pi_{m+1}^{c_{m+1}}=(n+1)^1$ to every permutation $\pi^{\textbf{c}}$ of $\S{_\textbf{m}^{\mathbf{r}}}$ (and there needs to be at least one bar following every descent). 
Moreover, when inserting the elements in between bars, we make sure that elements not separated by a bar are in non-decreasing order (to avoid forming descents between bars). 
These last two conditions guarantee that the permutations constructed satisfy the conditions of barred permutations. 

Every $t$-barred permutation arises uniquely in this way, i.e., by choosing $m_k$ out of $r_kt+1$ options (with repetition allowed). Let us realize by $ar_k+b$, the choice to place a $(b+1)$-colored copy of $k$ in the space after the $a$-th bar, where $0\leq a \leq t$ and $0\leq b \leq r_k-1$. Note that $a=0$ corresponds to the space right before the leftmost bar.
 Inserting an $l$-colored copy of $k$ in the space right after the $j$-th bar contributes the factor $z_{k}^{r_kj+l-1}$ to the weight of $\sigma\in B_t(\S{_\textbf{m}^{\mathbf{r}}})$. Moreover, the rightmost term, $\pi_{m+1}^{c_{m+1}}=(n+1)^1$, of $\pi^c$ contributes $z_{m+1}^t$ to the weight. 
Using the following $q$-binomial coefficient identity 
\begin{equation}\label{useful q-binomial identity} \begin{bmatrix}A+B \\  A \end{bmatrix}_{q} =\sum\limits_{0\leq g_1\leq g_2\leq...\leq g_A\leq B}q^{\sum_{i\in [A]} g_i}, \end{equation} \cite[Proposition 4.1]{FoataHan}, 
 we can describe the desired sum as follows:
\[\sum_{\sigma \in B_t(\S{_\textbf{m}^{\mathbf{r}}})}wt(\sigma)=z_{m+1}^t
\prod_{k\in[n]} 
\begin{bmatrix}r_kt +m_k \\  m_k \end{bmatrix}_{z_{k}}. \]

Summing over all possible values of $t$ completes the proof.
\end{proof}

For our study of the distributional properties of $A{{_{\textbf{m}}^{\textbf{r}}}}$, the relevant specialization of the identity in Theorem~\ref{thm: product formula multivariate} is given by $z_{m+1} = x$, $z_k = 1$, which yields the following corollary.

\begin{corollary}
    \label{cor: main id}
    For positive integral vectors $\mathbf{m}=(m_1,...,m_n)$ and $\mathbf{r}=(r_1,...,r_n)$, we have
    \[
    \frac{A{{_{\textbf{m}}^{\textbf{r}}}}}{(1 - x)^{m+1}} = \sum_{t\geq0}\left(\prod_{i=1}^n\binom{r_it + m_i}{m_i}\right)x^t.
    \]
\end{corollary}

\begin{rem}
    \label{rem: altproof}
We note that Corollary~\ref{cor: main id} may alternatively be proven using the identity \cite[Proposition~3.5(5)]{Branden2016LectureH} associated to the $s$-lecture hall order polytopes of Br\"and\'en and Leander. 
Specifically, Corollary~\ref{cor: main id} arises as a mild generalization of \cite[Corollary~3.6]{Branden2016LectureH}.
The proof included here shows that the direct combinatorial proof method of Stanley and Gessel \cite[Second Proof of Theorem~2.1]{Gessel} generalizes to colored multiset permutations.
\end{rem}

% Corollary~\ref{cor: main id} says that the colored multiset Eulerian polynomial $A{{_{\textbf{m}}^{\textbf{r}}}}$ is $\prod_{i=1}^n\binom{r_ix + m_i}{m_i}$ expressed in the multinomial basis for the vector space of univariate polynomials of degree at most $d$.
 {Note that the polynomial $\prod_{i=1}^n\binom{r_ix + m_i}{m_i}$ lives in the vector space of univariate polynomials of degree at most $d$ with standard basis $1, x, \ldots, x^d$.  Corollary~\ref{cor: main id} says that the coefficient sequence of the colored multiset Eulerian polynomial $A{{_{\textbf{m}}^{\textbf{r}}}}$ is given by the coefficients of $\prod_{i=1}^n\binom{r_ix + m_i}{m_i}$ when it is expressed in the multinomial basis $\binom{x + d}{d}, \binom{x + d - 1}{d}, \ldots, \binom{x}{d}$ for this vector space. }
This observation will allow us to derive real-rootedness and interlacing results on the $\mathcal{I}_m$-decomposition of $A{{_{\textbf{m}}^{\textbf{r}}}}$ in Section~\ref{sec:distributionalproperties}.

We now derive a $q$-analog in the case $\mathbf{r}=r\mathbf{1}$. For a colored permutation  
$\pi^{\textbf{c}}\in \mathfrak{S}{_\textbf{m}^{\textbf{r}}}$, we let \begin{itemize}
    \item $\text{dmaj}(\pi^{\mathbf{c}})=\sum\limits_{i\in DES(\pi^\mathbf{c})}(m-i)$,
    \item $\text{fdmaj}(\pi^{\mathbf{c}})
    =r\cdot \text{dmaj}(\pi^{\mathbf{c}})+\sum\limits_{i\in [m]}{(c_i-1)}$.
\end{itemize}

\begin{corollary}\label{cor: fmaj product formula multivariate}
 For positive integral vectors $\mathbf{m}=(m_1,...,m_n)$ and $\mathbf{r}=r\mathbf{1}$ with $r\geq 1$, we have
 
\begin{align}\label{first genera.} & 
\frac{\sum\limits_{\pi^{\textbf{c}} \in \S{_\textbf{m}^{\mathbf{r}}}} x^{\text{des}(\pi^{\mathbf{c}})} q^{\text{fdmaj}(\pi^{\mathbf{c}})}}{\prod\limits_{i=0}^{m}\big(1-q^{ri}x\big)}
 =\sum_{t\geq 0} x^t   
\prod_{k\in[n]} 
\begin{bmatrix}rt +m_k \\  m_k \end{bmatrix}_{q}.
\end{align}
\end{corollary}

\begin{proof}
The identity in \eqref{first genera.} can be seen as a specialization of the identity in Theorem~\ref{thm: product formula multivariate} for $z_{m+1}=x, z_k=q$, for $k\in [n]$, and $\mathbf{r}=r\mathbf{1}$.
\end{proof}

\subsubsection{Generating functions for ascents}
\label{subsubsec:ascents}
An identity analogous to \eqref{first genera} can be derived for ascents. We will use it to show that ascents and descents are equidistributed over the set of signed multiset permutations $\S{_\textbf{m}^{2\mathbf{1}}}$.

\begin{theorem}\label{thm: ascents MacMahon}
For positive integral vectors $\mathbf{m}=(m_1,...,m_n)$ and $\mathbf{r}=(r_1,...,r_n)$, we have

\begin{align}\label{eq: ascents MacMahon}
\sum_{\pi ^{\textbf{c}} \in \S{_\textbf{m}^{\mathbf{r}}}}\frac{
z_{\pi_1}^{r_{\pi_1}-c_{1}}\cdots z_{\pi_m}^{r_{\pi_m}-c_{m}}
\prod\limits_{\substack{i\in[m+1]\setminus\{1\} \\ i-1 \in ASC(\pi ^{\textbf{c}})}}z_{\pi_i}^{r_{\pi_i}}\cdots z_{\pi_m}^{r_{\pi_m}}z_{m+1}}{\prod\limits_{i\in [m+1]}\big(1-z_{\pi_i}^{r_{\pi_i}}\cdots z_{\pi_m}^{r_{\pi_m}}z_{m+1}\big)}
=\sum_{t\geq 0} z_{m+1}^t   
\prod_{k\in[n]} 
\begin{bmatrix}r_kt+(r_k-2) +m_k \\  m_k \end{bmatrix}_{z_{k}}.
\end{align}
\end{theorem}

The proof of Theorem \ref{thm: ascents MacMahon} is analogous to the proof of Theorem \ref{thm: product formula multivariate}, where the notion of barred permutations is replaced with its analog for ascents, i.e., where we replace descents by ascents in the definition. The corresponding weight of such an ascent-centered barred permutation on $\pi^{\mathbf{c}}\in \S{_\textbf{m}^{\mathbf{r}}}$ can be defined as
\[
z_{\pi_1}^{r_{\pi_1}-c_{1}}\cdots z_{\pi_m}^{r_{\pi_m}-c_{m}}
\prod\limits_{\substack{i\in[m+1] }}\big(z_{\pi_i}^{r_{\pi_i}}\cdots z_{\pi_m}^{r_{\pi_m}}z_{m+1}\big)^{d_{i}}.
\]
The term $(r_k-2)$ in \eqref{eq: ascents MacMahon} corresponds to the $(r_k-2)$ more colors that are permitted for an entry to the right of the last bar, compared to barred permutations. For example, for the colored permutation $\pi^{\mathbf{c}}=1^22^12^22^21^13^1$ of the multiset $M_\textbf{m}^{\mathbf{r}}$ with $\mathbf{m}=(2,3)$ and $\mathbf{r}=(3,3)$% with descents in positions $1$ and $4$
, the string $||1^22^1|2^2|2^3||1^{s}3^1$ satisfies the ascent analog of barred permutations for $s\in \{2,3\}$ but not for $s=1$.
However, only elements of color $1$ can follow the last bar in the descent version of barred permutations.

For a colored permutation  
$\pi^{\textbf{c}}\in \mathfrak{S}{_\textbf{m}^{\textbf{r}}}$, we let $\text{amaj}(\pi^{\mathbf{c}})=\sum\limits_{i\in ASC(\pi^{\mathbf{c}})}(m-i)$.
We will also use the notation $c_j(\pi^{\mathbf{c}})$ to denote the number of entries $i\in [m]$ of $\pi^{\mathbf{c}}$ that are $j$-colored for $j\in [r]$ where $r=\text{max}\{r_1,...,r_n\}$. 
For example, for $\pi^{\textbf{c}}=1^11^22^22^11^33^1$ we have $c_1(\pi^{\mathbf{c}})=2, c_2(\pi^{\mathbf{c}})=2, c_3(\pi^{\mathbf{c}})=1$.
With this notation, $\text{fdmaj}(\pi^{\mathbf{c}})=r \cdot \text{dmaj}(\pi^{\mathbf{c}})+\sum\limits_{j\in [r]}(j-1)c_j(\pi^{\mathbf{c}})$. Similarly, we define  $\text{famaj}(\pi^{\mathbf{c}}):=r\cdot \text{amaj}(\pi^{\mathbf{c}})+\sum\limits_{j\in [r]}(r-j)c_j(\pi^{\mathbf{c}})$.
Therefore, setting $z_{m+1}=x$ and $z_k=q$, for $k\in [n]$, in
Theorems \ref{thm: product formula multivariate} and \ref{thm: ascents MacMahon} leads to the following observation.

\begin{corollary}\label{cor: equidistribution}  For $\mathbf{m}=(m_1,...,m_n)$ and $\mathbf{r}=2\mathbf{1}$, we have
 
         \begin{align*}
\sum_{\pi^{\textbf{c}} \in \S{_\textbf{m}^{2\mathbf{1}}}} x^{\text{asc}(\pi^{\textbf{c}})}q^{\text{famaj}(\pi^{\mathbf{c}})}
=\sum_{\pi^{\textbf{c}} \in \S{_\textbf{m}^{2\mathbf{1}}}}
x^{\text{des}(\pi^{\textbf{c}})}q^{\text{fdmaj}(\pi^{\mathbf{c}})}.
\end{align*}
%        \begin{align*}
%\sum_{\pi^{\textbf{c}} \in \S_{M_\textbf{m}^{2\mathbf{1}}}} x^{\text{asc}(\pi^{\textbf{c}})}q^{2\text{amaj}(\pi^{\mathbf{c}})+c_1(\pi^{\mathbf{c}})}
%=\sum_{\pi^{\textbf{c}} \in \S_{M_\textbf{m}^{2\mathbf{1}}}}
%x^{\text{des}(\pi^{\textbf{c}})}q^{2\text{dmaj}(\pi^{\mathbf{c}})+c_2(\pi^{\mathbf{c}})}.
%\end{align*}
\end{corollary}

 {
\begin{rem}
    The equidistribution in Corollary \ref{cor: equidistribution} is also explained bijectively by mapping a signed multiset permutation $\pi_1^{c_1}\cdots\p_m^{c_m}(n+1)^1\in \S{_\textbf{m}^{2\mathbf{1}}}$ to the signed multiset permutation 
    ${\pi_{1}'}^{{c_{1}'}}\cdots{\pi_{m}'}^{{c_{m}'}}(n+1)^{1}$, where 
    $\pi'_i=n+1-\pi_i$ and $c'_i={3-c_i}$ for $i\in [m]$.
    \end{rem}}
 \begin{rem}
For a permutation $\pi^{\mathbf{c}}=\pi_1^{c_1}\cdots\p_m^{c_m}(n+1)^1$, let  $\text{rev}(\pi^{\mathbf{c}}):=(n+1)^1\p_m^{c_m}\cdots\pi_1^{c_1}$ denote the {\em reverse permutation} of $\pi^{\mathbf{c}}$. An ascent (descent) $i\in [m]$ of $\pi^{\mathbf{c}}$ corresponds to a descent (ascent) $m-i$ of $\text{rev}(\pi^{\mathbf{c}})$ and vice versa. If we reverse our permutations, the equidistribution between ascents and descents in Corollary \ref{cor: equidistribution} implies that Theorems  \ref{thm: product formula multivariate} and \ref{thm: ascents MacMahon} generalize a known formula for signed permutations (Theorem 3.7 in \cite{Chow2007OnTD}) and its multiset analog (equation (3.3) in \cite{LZC}) proven using a different coloring order.
 \end{rem}

\subsection{A second multivariate identity}
\label{subsec:identity2}
We now proceed to derive an alternative multivariate extension of the identity~\eqref{macmahon's formula}, which generalizes Lin's equation \eqref{lin's formula}. 
To do so, we introduce the additional variables $y_j$ for $j\in [r]$ where $r:=\text{max}\{r_1,...,r_n\}$. Recall that $c_j(\pi^{\mathbf{c}})$ is the number of $j$-colored entries in $\pi^{\mathbf{c}}$ (excluding $\pi_{m+1}^{1}$).
We then have the following result, where $\lfloor a \rfloor$ denotes the greatest integer less than or equal to $a$  {and the notation $i_{l;k}$ is used to specify that the index corresponds to a chosen number $k$ and a color $l$ from the color set $[r_k]$}.

\begin{theorem}
\label{mcmahon general(2)} 
For positive integral vectors $\mathbf{m}=(m_1,...,m_n)$ and $\mathbf{r}=(r_1,...,r_n)$, we have

    \begin{align}\label{mc mahon generali2}
 & \sum_{\pi^{\textbf{c}} \in \S{_\textbf{m}^{\mathbf{r}}}}\frac{\prod\limits_{p\in [r]}y_p^{c_{p}(\pi^{\textbf{c}})} 
\prod\limits_{\substack{i\in[m+1]\setminus\{1\} \\ i-1 \in \text{DES}(\pi^{\textbf{c}})}}z_{\pi_i^{c_i}}^{r_{\pi_i}}\cdots z_{\pi_m^{c_m}}^{r_{\pi_m}}z_{m+1}}{\prod\limits_{i\in [m+1]}\big(1-z_{\pi_i^{c_i}}^{r_{\pi_i}}\cdots z_{\pi_m^{c_m}}^{r_{\pi_m}}z_{m+1}\big)}
\\ &= \sum_{t\geq 0} z_{m+1}^t 
\prod_{k\in[n]} \bigg(
\sum_{\footnotesize{\substack{i_{r_k;k}\in [m_k]_0 \\
    \color{red}{\vdots}\\
    i_{2;k} \in [m_k-\sum\limits_{l=3}^{r_k}i_{l;k}]_0 }}}
    \Big(
\prod\limits_{l\in [r_k]} \begin{bmatrix}t+\lfloor\frac{1-l}{l} \rfloor +i_{l;k} \\  i_{l;k} \end{bmatrix}_{z_{k^l}^{r_k}}  y_l^{i_{l;k}}\Big)\bigg),
\end{align}
     where $i_{1;k}=m_k-\sum\limits_{l=2}^{r_k}i_{l;k}$ in the sum, for each $k\in [n]$.
\end{theorem}

\begin{proof}
The proof is similar to that of Theorem \ref{thm: product formula multivariate} (also based on \cite[Theorem 6]{LZC}).  For a barred permutation $\sigma \in B(\S{_\textbf{m}^{\mathbf{r}}})$ on $\pi^{\mathbf{c}} \in \S{_\textbf{m}^{\mathbf{r}}}$,
where $d_i$ is the number of bars between $\pi_{i-1}^{c_{i-1}}$ and $\pi_i^{c_i}$ for $i=2,...,m+1$ and $d_{1}$ the number of bars to the left of $\pi_1^{c_1}$, let the weight of $\sigma$ be

\[wt(\sigma):=\prod\limits_{p\in [r]}y_p^{c_{p}(\pi^{\mathbf{c}})} 
\prod\limits_{\substack{i\in[m+1] }}\big(z_{\pi_i^{c_i}}^{r_{\pi_i}}\cdots z_{\pi_m^{c_m}}^{r_{\pi_m}}z_{m+1}\big)^{d_{i}}
.\]

For example, for the barred permutation $\sigma= ||1^2|2^12^2|2^2||1^13^1$ on $\pi^{\mathbf{c}}\in  \S{_\textbf{m}^{\mathbf{r}}}$ for $\textbf{m}=(2,3)$ and a vector of positive integers $\mathbf{r}=(r_1,r_2)$,
we compute \[wt(\sigma)=\ 
y_1^2y_2^3(z_{1^2}^{r_1}z_{2^1}^{r_2}z_{2^2}^{2r_2}z_{1^1}^{r_1}z_6)^2(z_{2^1}^{r_2}z_{2^2}^{2r_2}z_{1^1}^{r_1}z_6)(z_{2^2}^{r_2}z_{1^1}^{r_1}z_6)(z_{1^1}^{r_1}z_6)^2\ =\ y_1^2y_2^3z_{1^2}^{2r_1}z_{2^1}^{3r_2}z_{2^2}^{7r_2}z_{1^1}^{6r_1}z_6^6.\] We will now sum over the weights of all barred permutations in $B(\S{_\textbf{m}^{\mathbf{r}}})$ in two ways.

First, we fix a permutation $\pi^{\mathbf{c}} \in \S{_\textbf{m}^{\mathbf{r}}}$ and sum over the set $B(\S{_\textbf{m}^{\mathbf{r}}})(\pi^{\mathbf{c}})$ of barred permutations on $\pi^{\mathbf{c}}$. Similarly to the proof of Theorem \ref{thm: product formula multivariate}, we compute

\begin{align*}
& \sum_{\sigma \in B(\pi^{\mathbf{c}})}wt(\sigma)
 =\frac{\prod\limits_{p\in [r]}y_p^{c_{p}(\pi^{\mathbf{c}})} 
\prod\limits_{\substack{i\in[m+1]\setminus\{1\} \\ i-1 \in DES(\pi^{\mathbf{c}})}}z_{\pi_i^{c_i}}^{r_{\pi_i}}\cdots z_{\pi_m^{c_m}}^{r_{\pi_m}}z_{m+1}}{\prod\limits_{i\in [m+1]}\big(1-z_{\pi_i^{c_i}}^{r_{\pi_i}}\cdots z_{\pi_m^{c_m}}^{r_{\pi_m}}z_{m+1}\big)}.
\end{align*}

Let us now restrict to the set $B_t(\S{_\textbf{m}^{\mathbf{r}}})$ of barred permutations with $t$ bars. We saw in the proof of Theorem \ref{thm: product formula multivariate} that we can construct such permutations
by first placing $t$ bars and then inserting the colored colored copies of $k$ for all $k\in [n]$ in a suitable way.

Inserting $m_k$ elements from $\{k^1,k^2,...,k^{r_k}\}$ (with repetition allowed), is the same as selecting $i_{r_k;k}$ copies of $k^{r_k}$ to insert, $i_{r_k-1;k}$ copies of $k^{r_k-1}$, and so on, up to $i_{2;k}$ copies of $k^{2}$, and lastly $i_{1;k}:=m_k-i_{2;k}-\cdots -i_{r_k;k}$ copies of $k^{1}$.
Each of the $i_{1;k}$ copies of $k^1$ can be placed in the $j$-th space for any $j=0,...,t$, and each copy of $k^2,...,k^{r_k}$ can be placed in the $j$-th space for any $j=0,...,t-1$ ($j=0$ denotes the leftmost space, before the first bar).

Inserting an $l$-colored copy of $k$ in the space right after the $j$-th bar, contributes the factor $y_lz_{k^l}^{r_kj}$ to the weight of $\sigma\in B_t(\S{_\textbf{m}^{\mathbf{r}}})$.
Moreover, the term $\pi_{m+1}^{c_{m+1}}=(n+1)^1$ of $\pi^{\mathbf{c}}$ contributes $z_{m+1}^t$ to the weight. 
Using the $q$-binomial coefficient identity in \eqref{useful q-binomial identity}, 
it follows
that
the weight of a permutation $\sigma \in B_t(\S{_\textbf{m}^{\mathbf{r}}})$ on $\pi^{\mathbf{c}}$
for a fixed choice of $i_{l;k}$ ($1\leq l\leq r_k, 1\leq k \leq n$) is equal to 
\[
z_{m+1}^t\prod\limits_{k\in [n]}  
\Bigg(\Big(  \begin{bmatrix}t +i_{1;k} \\  i_{1;k} \end{bmatrix}_{z_{k^1}^{r_k}}  y_1^{i_{1;k}}\Big)
\prod\limits_{l\in [r_k]\setminus \{1\}} \Big(  \begin{bmatrix}t-1 +i_{l;k} \\  i_{l;k} \end{bmatrix}_{z_{k^l}^{r_k}} y_l^{i_{l;k}}\Big)
\Bigg).
\]

Summing over all possible choices of $i_{l;k}$ gives

\[\sum_{\sigma \in B_t(\S{_\textbf{m}^{\mathbf{r}}})}wt(\sigma)=
z_{m+1}^t 
\prod_{k\in[n]} \bigg(
\sum_{\footnotesize{\substack{i_{r_k;k}\in [m_k]_0 \\
    \cdots\\
    i_{2;k} \in [m_k-\sum\limits_{l=3}^{r_k}i_{l;k}]_0 }}}
    \Big(\begin{bmatrix}t +i_{1;k} \\  i_{1;k} \end{bmatrix}_{z_{k^1}^{r_k}} y_1^{i_{1;k}}
\prod\limits_{l\in [r_k]\setminus\{1\}} \begin{bmatrix}t-1 +i_{l;k} \\  i_{l;k} \end{bmatrix}_{z_{k^l}^{r_k}}  y_l^{i_{l;k}}\Big)\bigg),\]
 where $i_{1;k}=m_k-\sum\limits_{l=2}^{r_k}i_{l;k}$ in the sum, for each $k\in [n]$. The proof follows by summing over all possible values of $t$, and using that $\lfloor \frac{1-a}{a}\rfloor=0$ for $a=1$ and $\lfloor \frac{1-a}{a}\rfloor=-1$ for $a>1$.
\end{proof}

Setting $z_{m+1}=x$, $y_l=1$, $z_{k^l}=1$ and $r_k=1$ for all $l \in [r_k], k\in [n]$ in the identity in Theorem \ref{mcmahon general(2)} recovers \eqref{macmahon's formula}. 
Similarly, equation \eqref{lin's formula} is recovered by setting $z_{m+1}=x$, $y_l=1$, $z_{k^l}=1$ and $r_k=2$ for all $l \in [r_k], k\in [n]$. 
More generally, by Corollary \ref{cor: equidistribution}, Theorem \ref{mcmahon general(2)} provides a generalization of equation (3.2) in \cite{LZC}.
We record the following specialization, arising for $r_k=r$, $z_{m+1}=x$, $y_l=q^{l-1}$ and $z_{k^l}=q$ and for all $l \in [r], k\in [n]$.

\begin{corollary}\label{cor: second identity q-analog}
For positive integral vectors $\mathbf{m}=(m_1,...,m_n)$ and $\mathbf{r}=(r_1,...,r_n)$, we have

    \begin{align}\label{eq: second identity corollary }
 & \frac{ \sum\limits_{\pi^{\textbf{c}} \in \S{_\textbf{m}^{\mathbf{r}}}}x^{\text{des}(\pi^{\mathbf{c}})} q^{\text{fdmaj}(\pi^{\mathbf{c}})}}{\prod\limits_{i=0}^{m}\big(1-q^{ri}x\big)}
 = \sum_{t\geq 0} x^t 
\prod_{k\in[n]} \bigg(
\sum_{\footnotesize{\substack{i_{r;k}\in [m_k]_0 \\
    \cdots\\
    i_{2;k} \in [m_k-\sum\limits_{l=3}^{r}i_{l}]_0 }}}
    \Big(
\prod\limits_{p\in [r]} \begin{bmatrix}t+\lfloor \frac{1-p}{p} \rfloor +i_{p;k} \\  i_{p;k} \end{bmatrix}_{q^{r}}  q^{(p-1)i_{p;k}}\Big)\bigg),
\end{align}
     where $i_{1;k}=m_k-\sum\limits_{l=2}^{r_k}i_{l;k}$ in the sum, for each $k\in [n]$.
\end{corollary}

Note that the left-hand-side in \eqref{eq: second identity corollary } is equal to the left-hand-side in \eqref{first genera.}. However, \eqref{eq: second identity corollary } is more general in the sense that all positive color vectors $\mathbf{r}=(r_1,...,r_n)$ are permitted.

\subsection{A third multivariate identity.}

In this subsection, we study a different generalization of MacMahon's identity, which recovers an identity due to Carlitz, and some of its multivariate analogs studied by Beck and Braun in the case of colored permutations and by Lin in the case of signed multipermutations.

\begin{definition}[\cite{eulermah}]
For a colored multiset permutation $\pi^{\mathbf{c}}=\pi_1^{c_1}...\pi_{m}^{c_m}(n+1)^1 \in \S{_\textbf{m}^{\mathbf{r}}}$, we define
\[a_j({\pi^{\mathbf{c}}})=(c_j-c_{j+1})\textrm{ mod }r\]
to be the {\em $j$-th color change} of $\pi^{\mathbf{c}}$, $j=1,...,m$.
\end{definition}

In \cite{LZC}, Lin introduced the notion of flag barred permutations for signed multiset permutations to generalize combinatorial identities such as those by Beck and Braun in \cite{eulermah}. We extend this approach to colored multiset permutations. Recall that $r:=\text{max}\{r_1,...,r_n\}$.

\begin{definition}\label{flag stuff} 
For $\pi^{\mathbf{c}}=\pi_1^{c_1}\cdots \pi_m^{c_m} (n+1)^1\in \S{_\textbf{m}^{\mathbf{r}}}$, a {\em flag barred permutation} $\sigma$ on $\pi^{\mathbf{c}}$ is obtained by inserting bars between the entries of $\pi^{\mathbf{c}}$ such that
\begin{itemize}
    \item[i)]  every descent space $j$ of $\pi^{\mathbf{c}}$ with $a_j(\pi^{\mathbf{c}})=0$, $j=1,...,m$, receives at least $r$ bars,
    \item[ii)] the number of bars inserted between the $j$-th and the $(j+1)$-th entry of $\pi^{\mathbf{c}}$ is equivalent to $a_j(\pi^{\mathbf{c}})$ modulo $r$,
    \item[iii)] any number of bars is inserted in the $0$-th space to the left of $\pi_1^{c_1}$.
\end{itemize}
    We denote by $ F(\S{_\textbf{m}^{\mathbf{r}}})$ the set of flag barred permutations on $\S{_\textbf{m}^{\mathbf{r}}}$.    By definition, each decent space of $\pi^{\mathbf{c}}$ will receive at least one bar, and hence $ F(\S{_\textbf{m}^{\mathbf{r}}})\subset  B(\S{_\textbf{m}^{\mathbf{r}}})$.
\end{definition}

For example, $|1^1||1^22^2|2^1||1^3||3^1$  is a barred permutation on $\p^{\mathbf{c}}=1^11^22^22^11^33^1 \in \S{_{\textbf{m}}^{\textbf{r}}}$ with $\textbf{m}=\textbf{r}=(3,2)$ that is not flag, whilst $|1^1||1^22^2|2^1||||1^3||3^1$ is flag.

The following theorem generalizes \cite[Theorem 19]{LZC} and \cite[Theorem 15.5]{eulermah}. 
\begin{theorem}
\label{latest generalization} 
For positive integral vectors $\mathbf{m}=(m_1,...,m_n)$ and $\mathbf{r}=r\mathbf{1}$, $r\geq 1$, the following identity holds:
\begin{equation}\label{eq: 2second product formula}
\sum\limits_{\pi^{\textbf{c}} \in \S{_\textbf{m}^{\mathbf{r}}}}\frac{
\prod\limits_{\substack{i\in[m] }}(z_0z_{\pi_1}\cdots z_{\pi_i})^{a_i(\pi^{\mathbf{c}})} \prod\limits_{\substack{a_i(\pi^{\mathbf{c}})=0\\ i \in \text{Des}(\pi^{\mathbf{c}})}}(z_0z_{\pi_1}\cdots z_{\pi_i})^r }
{(1-z_0)  \prod\limits_{i\in[m] }\big(1-(z_0z_{\pi_1}\cdots z_{\pi_i})^{r}\big)}=\sum\limits_{t\geq 0} z_0^t
\prod\limits_{k\in [n]}\begin{bmatrix}t+m_k \\ m_k\end{bmatrix}_{z_{k}}.\end{equation}
\end{theorem}
\begin{proof}
 
For a flag barred permutation $\sigma \in F(\S{_\textbf{m}^{\mathbf{r}}})$ on $\pi^{\mathbf{c}}$, we let $b_i (=d_{i+1})$ be the number of bars in the $i$-th space, between $\pi_i^{c_i}$ and $\pi_{i+1}^{c_{i+1}}$ $(i=1,...,m)$ (where $b_0$ counts the bars left of $\pi_1^{c_1}$), and define the weight of $\sigma$ as
\[wt_F(\sigma)=\prod\limits_{i\in [m]}(z_0z_{\pi_1}\cdots z_{\pi_i})^{b_i}. \]
For example for $\sigma =|1^1||1^22^2|2^1||||1^3||3^1$, $wt_F(\sigma)=z_0^{10}z_{1}^{18}z_{2}^{13}$. 
We will sum over the weights of flag barred permutations in $F(\S{_\textbf{m}^{\mathbf{r}}})$ in two ways.

First, we fix a permutation $\pi^{\mathbf{c}} \in \S{_\textbf{m}^{\mathbf{r}}}$. A flag barred permutation $\sigma$ on $\pi^{\mathbf{c}}$ can be obtained by inserting $a_j(\pi^{\mathbf{c}})$ bars in each position $j\in [m]$ $-$ and $r$ bars in case  $a_j(\pi^{\mathbf{c}})=0$ and $j$ is 
a descent $-$
and then inserting any number of $r$-ples of bars in each of the $m$ spaces $1,...,m$. Lastly, any number of bars can be placed in the $0$-th space. All flag barred permutations on $\pi^{\mathbf{c}}$ arise (uniquely) in this way. Therefore, summing over the weights of all possible flag barred permutations $\sigma$ of $\pi^{\mathbf{c}}$ yields
\begin{align*} &
\sum_{\sigma \in F(\pi^{\mathbf{c}})}wt_F(\sigma)
 =\frac{\prod\limits_{\substack{i\in[m] }}(z_0z_{\pi_1}\cdots z_{\pi_i})^{a_i(\pi^{\mathbf{c}})} \prod\limits_{\substack{a_i(\pi^{\mathbf{c}})=0\\ i \in \text{Des}(\pi^{\mathbf{c}})}}(z_0z_{\pi_1}\cdots z_{\pi_i})^r }
{(1-z_0)  \prod\limits_{i\in[m] }\big(1-(z_0z_{\pi_1}\cdots z_{\pi_i})^{r}\big)}.
\end{align*}

Now, let us instead partition $F(\S{_\textbf{m}^{\mathbf{r}}})$ into sets of flag barred permutations with exactly $t$ bars, denoted by $F_t(M_\textbf{m}^{\mathbf{r}})$, $t\geq 0$. 
Each flag $t$-barred permutation $\sigma \in F_t(M_\textbf{m}^{\mathbf{r}})$ can be constructed by inserting $m_k$ copies of $k$, for $k\in [n]$, to the $t+1$ spaces between adjacent bars, including the spaces on the left and right side, that we then color in accordance with the definition of flag barred permutations. 
Specifically, any element placed at the right side of the rightmost bar is chosen to be $1$-colored, since we have fixed $\pi_{m+1}^{c_{m+1}}=(n+1)^1$. 
Similarly, all entries between the $(t-j)$-th and the $(t+1-j)$-th bar (counting from left to right) receive the color $(j \text{ mod }r)+1$, for $1\leq j \leq t$ (here $j=t$ refers to the space before the 1st bar). 
The elements between bars are again ordered in weakly increasing order so as to not create any descents between bars.  The color of $\pi_i^{c_i}$ is $c_i$, where $c_i-1$ is equivalent modulo $r$ to the number of bars to the right of $\pi_i$. 
We denote the number of bars to the right of $\pi_i$ as $j_{i}$, for $0\leq j_i\leq t$. 
For two consecutive entries $\pi_i^{c_i}, \pi_{i+1}^{c_{i+1}}$ of $\pi^{\mathbf{c}}$,  we compute $a_i({\pi}^{\mathbf{c}})=(c_i-c_{i+1})\text{ mod }r\equiv (j_i-j_{i+1})\text{ mod }r$, which means that the barred permutation $\sigma$ constructed satisfies condition ii) of Definition \ref{flag stuff}. 
Moreover, the entries of $\sigma$ between bars are in weakly increasing order, respecting condition i) of Definition \ref{flag stuff}. 
Hence, $\sigma \in F_t(M_\textbf{m}^{\mathbf{r}})$.

To show that every flag barred permutation arises through this construction, let us start with a flag barred permutation $\sigma \in F_t(M_\textbf{m}^{\mathbf{r}})$ on $\pi^{\mathbf{c}} \in \S{_\textbf{m}^{\mathbf{r}}}$.
By definition, since $\pi_{m+1}^{c_{m+1}}=(n+1)^1$, it follows that $a_m(\pi^{\mathbf{c}})\equiv (c_m-1) \text{ mod }r $.
Hence the number of bars to the right of $\pi_m^{c_m}$ is  $b_{m}=c_m-1+rs$, for some $s\geq 0$ such that $c_m-1+rs\leq t$. 
Similarly, since $\sigma$ is flag, the number of bars in the space between $\pi_{m-1}^{c_{m-1}}$ and $\pi_{m}^{c_m}$ is $b_{m-1}=a_{m-1}(\pi^{\textbf{c}})+s'r $, for some $s'\geq 0$. 
Hence, the total number of bars that are located to the right of $\pi_{m-1}^{c_{m-1}}$ modulo $r$ is $a_{m-1}(\pi^{\textbf{c}})+a_m(\pi^{\textbf{c}}) %\equiv l_m-1+(l_{m-1}-l_m) 
\equiv (c_{m-1}-1 )\text{ mod } r$. 
In general, the total number of bars to the right of $\pi_{i}^{c_i}$, for $i\in [m]$, is $b_i+\cdots+b_{m}\equiv (c_{i}-1) \text{ mod }r$ hence $c_i\equiv( b_i+\cdots+b_{m}+1) \text{ mod }r$.
Since $1\leq c_i\leq r$, it follows that $c_i=b+1$, where $b_i+\cdots+b_{m}\equiv b \text{ mod }r$.
It follows that the construction described above will give rise to $\sigma$ since we saw that the color $c_i$ of $\pi_i^{c_i}$ respects the construction restrictions.

Let us now count the number of flag barred permutations in $F_t(M_\textbf{m}^{\mathbf{r}})$ for a fixed number of bars $t$. Starting with $t$ bars, we can construct a flag barred permutation by adding $m_k$ colored copies of $k$, in any order, with the restriction that $l$-colored copies of $k$ are placed in the space after the $(t+1-l-sr)$-th bar (counting from left to right) for some $s\geq 0$ with $t+1-l-sr \geq 0$. 
For each $m_k$-tuple of colored copies of $k$ that we place in the spaces between the $t$ bars, there are $t+1$ choices for the space (which also determines color) of each copy of $k$. Since  the $t$ bars contribute $z_0^t$ to the weight, and each copy of $k$ inserted in the space after the  $(t-j)$-th bar contributes the term
$z_{k}^{j}$,
by \eqref{useful q-binomial identity}, the weight of  $\sigma \in F_t(M_\textbf{m}^{\mathbf{r}})$ is

\[wt_F(\sigma)= z_0^t
\prod\limits_{k\in [n]}   \begin{bmatrix} t +m_k \\ m_k\end{bmatrix}_{z_{k}}.\] 
Summing over all values of $t$ proves the second equality.
\end{proof}

As mentioned, Theorem \ref{latest generalization} provides a colored analog of equation (5.10) in \cite{LZC}, which is itself a generalization of an identity due to Foata and Han - equation (7.3) in \cite{Foata2005} as noted in \cite{LZC}. Lin's equation  (5.10) arises for $r=2$.
The specialization $r=1, m_k=1$ and $z_{\pi_i}=q$ for $k\in [n], i \in [m]$, recovers a formula due to Carlitz \cite{carlitzidentity} (Theorem 1.3 in \cite{eulermah}). It is also worth noting that Bagno and Biagioli \cite{bagnobiagioli} (Theorem 5.2 in \cite{eulermah}) established a colored analog of Carlitz's formula, which is recovered from equation \eqref{eq: 2second product formula} by setting
$m_k=1$ and $z_{\pi_i}=q$ for $k\in [n], i \in [m]$. In \cite{eulermah},
 Beck and Braun provided a multivariate generalization \cite[Theorem 4.1]{eulermah} of Carlitz's formula, which they also generalized to colored permutations \cite[Theorem 5.15]{eulermah}. The latter theorem arises from equation \eqref{eq: 2second product formula} by setting $m_k=1$ for $k\in [n]$ and \cite[Theorem 4.1]{eulermah} by further specializing $r=1$\footnote{We note that the definition of descents used in \cite[Theorem 5.15]{eulermah} coincides with our definition of descents in case there is no color change.}.

\begin{rem}
 An interesting phenomenon regarding equation \eqref{eq: 2second product formula}  is that the value on the right handside  does not depend on the value of $r$.
\end{rem}

 While Theorem \ref{latest generalization} is stated for $\mathbf{r}=r\mathbf{1}$, one can derive a formula for general $\mathbf{r}$. We do not present it here because it is rather technical.

\section{Distributional properties}
\label{sec:distributionalproperties}

In this section, we describe the distributional properties of the colored multiset Eulerian polynomials. 
In Subsection~\ref{subsec:symmetric}, we characterize when the colored-multiset Eulerian polynomial $A_{\textbf{m}}^{\textbf{r}}$ is symmetric. 
In Subsection~\ref{subsec:self-interlacing}, we show that $A_{\textbf{m}}^{\textbf{r}}$ is self-interlacing whenever $r_i \geq m_i + 1$ for all $i$, generalizing the results of \cite{PS} for colored Eulerian polynomials. It follows that these polynomials satisfy several well-studied distributional properties.

\subsection{Preliminaries}
\label{sec: preliminaries}
We recall the definitions of the distributional properties of interest and their basic properties.
For a more detailed discussion of these properties and their significance in combinatorics, we recommend the survey article \cite{PB}. 

A polynomial $p=p_0+p_1x+\cdots+p_dx^d$ of degree $d$ is \emph{unimodal} if its coefficients satisfy $p_0\leq \cdots \leq p_s\geq \cdots \geq p_d$ for some $s\in \{0,...,d\}$. %; $s$ is then a {\em mode} of $p(x)$. 
It is \emph{log-concave} if $p_i^2 \geq p_{i-1}p_{i+1}$ for all $i = 1,\ldots, d$. 
It is well-known that a log-concave polynomial with no internal zeros is also unimodal. 

The polynomial $p$ is called {\em symmetric} with respect to degree $n$ if $p_s=p_{n-s}$ for all $s=0,...,n$. 
The linear space of polynomials that are symmetric with respect to degree $n$ can be expressed in the basis
\[
\{x^i(x + 1)^{n - 2i} : 0 \leq i \leq \lfloor n/2\rfloor\},
\]
which is commonly referred to as the \emph{$\gamma$-basis}.  
If a polynomial has nonnegative coefficients in the $\gamma$-basis it is called \emph{$\gamma$-positive} (or, more accurately, \emph{$\gamma$-nonnegative}). 
If $p$ is $\gamma$-positive it follows that $p$ is unimodal. 

The polynomial $p$ is said to be {\em real-rooted} if $p$ is a constant polynomial or all of its zeros are real numbers. 
When the coefficients of $p$ are nonnegative, it follows that $p$ is log-concave.  
Moreover, when $p$ is real-rooted and symmetric, $p$ is also $\gamma$-positive.  
Hence, a proof that $p$ with only positive coefficients is real-rooted and symmetric is a proof that $p$ satisfies all of the above distributional properties. 

The following distributional property is currently a focal point in algebraic combinatorics and discrete geometry:
The polynomial $p$ is {\em alternatingly increasing} if  $p_0\leq p_d \leq p_1 \leq p_{d-1}\leq \cdots \leq p_{\lfloor \frac{d+1}{2}\rfloor}$. 
If $p$ is alternatingly increasing, it is also unimodal. 
This property is of interest in algebraic combinatorics since the following characterization makes proofs of the alternatingly increasing property useful when studying the unimodality of Hilbert polynomials of graded rings expressed in the multinomial basis. 

Given $n\geq d$, it can be shown that there exist unique polynomials $a,b\in\mathbb{R}[x]$ such that
(1) $p = a + xb$, (2) $\deg(a) \leq n$, (3) $\deg(b)\leq n-1$, (4) $a$ is symmetric with respect to $n$, and (5) $b$ is symmetric with respect to $n-1$.
The pair $(a,b)$ is called the \emph{symmetric decomposition} of $p$ with respect to $n$, or the \emph{$\mathcal{I}_n$-decomposition} of $p$. 
A basic observation is that $p$ is alternatingly increasing if and only if both $a$ and $b$ have only nonnegative coefficients and are unimodal.
It is therefore of interest to investigate the distributional properties of the polynomials in a symmetric decomposition. 
We say that $p$ has a nonnegative, real-rooted, unimodal, or log-concave symmetric decomposition whenever both $a$ and $b$ fulfill the specified condition. 
When $a$ and $b$ are both $\gamma$-positive $p$ is said to be \emph{bi-$\gamma$-positive}.

In \cite{PS} it was shown that a sufficient condition for $p$ to have a real-rooted symmetric decomposition (when the decomposition is nonnegative) is for $p$ to be \emph{interlaced} by its own reciprocal.  
Given two real-rooted polynomials $p$ and $q$ with respective zeros $\alpha_1\geq \alpha_2\geq \cdots$ and $\beta_1\geq \beta_2\geq \cdots$, we say that $p$ is \emph{interlaced} by $q$, denoted $q\preceq p$ if
\[
\alpha_1\geq \beta_1\geq \alpha_2\geq \beta_2\geq \cdots.
\]
When these inequalities are strict, we instead write $q\prec p$.
The reciprocal of the polynomial $p$ (with respect to degree $d$) is $\mathcal{I}_d(p) = x^dp(1/x)$. 
A polynomial $p$ is \emph{self-interlacing} (with respect to $d$) if $\mathcal{I}_d(p) \preceq p$. 
Note that self-interlacing polynomials are a natural generalization of symmetric, real-rooted polynomials.

In \cite{PS}, it is shown that a polynomial $p$ of degree at most $d$ with $\mathcal{I}_d$-decomposition $(a,b)$ having only nonnegative coefficients is (strictly) self-interlacing with respect to $d$ if and only if $b\prec a$; specifically, we have the following theorem:
\begin{theorem}
    \cite[Theorem 2.6]{PS}
    \label{thm:BL}
    Let $p$ be a univariate polynomial of degree at most $d$ with nonnegative $\mathcal{I}_d$-decomposition $(a,b)$. The following are equivalent:
    \begin{enumerate}
        \item $b \prec a$,
        \item $b \prec p$,
        \item $a \prec p$,
        \item $\mathcal{I}_d(p) \prec p$, and
        \item $\mathcal{R}_d(p) \prec p$, where $\mathcal{R}_d(p) = (-1)^dp(-1 - x)$.
    \end{enumerate}
\end{theorem}
In the case that any one of (1) - (5) is satisfied, it follows that the symmetric decomposition of $p$ is real-rooted, log-concave, and unimodal, but we also recover that $p$ is alternatingly increasing, real-rooted, log-concave, unimodal and bi-$\gamma$-positive.  
Hence, it is of interest to determine when a given combinatorial generating polynomial is self-interlacing (e.g., Theorem~\ref{thm:BL}(4)).  
This property was shown to hold, for example, for colored Eulerian polynomials and applied to settle several conjectures in algebraic combinatorics \cite{PS}.  
In the following, we describe when this condition generalizes to the larger family of colored multiset Eulerian polynomials.

\begin{example}
    \label{ex: non-real-rooted with real-rooted dec}
    Note that it is possible for a polynomial that is not real-rooted to have a real-rooted symmetric decomposition with respect to its degree.  
    For instance, the polynomial
    $
    p = 1+5x+17x^2+15x^3+2x^4
    $
 has two non-real zeros but $\mathcal{I}_4$-decomposition $((1+x)^4, A_4(x))$. 
    Moreover, since $p$ has non-real zeros, it is clearly not interlaced by its own reciprocal. 
    Equivalently, we note that $A_4(x)$ does not interlace $(1+x)^4$. 
\end{example}

\begin{example}
    \label{ex: real-rooted but non-real-rooted dec}
    Similarly, it is possible for a polynomial to be real-rooted and not have a real-rooted symmetric decomposition with respect to its degree.  
    For instance, $p=(1+2x)^2$ has $\mathcal{I}_2$-decomposition $(1+x+x^2, 3A_2(x))$.
\end{example}

\subsection{Symmetric Colored Multiset Eulerian Polynomials}
\label{subsec:symmetric}

We now consider symmetry of the colored multiset Eulerian polynomials $A{_\mathbf{m}^{\mathbf{r}}}$ for positive integral vectors $\textbf{m} = (m_1,\ldots, m_n)$ and $\textbf{r} = (r_1,\ldots,r_n)$. 
To do so, we make use of the identity in Corollary~\ref{cor: main id}, noting first that it has a simple geometric interpretation. 

The \emph{Ehrhart polynomial} of a $d$-dimensional convex lattice polytope $P\subset\mathbb{R}^n$ is $\mathcal{L}(P; t) = |tP \cap\mathbb{Z}^n|$, where $tP = \{tp : p\in P\}$ denotes the \emph{$t$-th dilate} of $P$ for $t\in\mathbb{Z}_{>0}$. 
The \emph{(Ehrhart) $h^\ast$-polynomial} of $P$, denoted $h^\ast(P; x)$, is the Ehrhart polynomial of $P$ when expressed in the multinomial basis for the vector space of polynomials of degree at most $d$.  That is, $h^\ast(P; x)$ is the polynomial satisfying
\[
1 + \sum_{t > 0}\mathcal{L}(P; t)x^t = \frac{h^\ast(P; x)}{(1 - x)^{d + 1}}.
\]
In \cite{stanleydecomp}, Stanley showed that $h^\ast(P; x)$ has only nonnegative integer coefficients with constant term $1$. 
It is also well-known that $\mathcal{L}(\Delta_d;t)= \binom{t + d}{d}$ when $\Delta_d$ is the the standard $d$-dimensional simplex \cite[Theorem 2.2]{BR}.
Since $\mathcal{L}(rP;t) = \mathcal{L}(P; rt)$ for any positive integer $r$ and taking a direct product of polytopes corresponds to multiplying the Ehrhart polynomials, we have that the polynomial $\prod_{i=1}^n\binom{r_it + m_i}{m_i}$ appearing in Corollary~\ref{cor: main id} is the Ehrhart polynomial of the product of dilated simplices
\begin{equation}\label{eqn:simprod}
P^{\mathbf{r}}_{\mathbf{m}}:=\prod_{j=1}^{ n}r_j\Delta_{m_j}.
\end{equation}
It follows that $A{_{\textbf{m}}^{\textbf{r}}}$ is the $h^\ast$-polynomial of the $m$-dimensional polytope $P^{\mathbf{r}}_{\mathbf{m}}$, where $m = m_1 + \cdots + m_n$, as stated in the next lemma. 
\begin{lemma}\label{lem: colored eulerian multiset permutation polynomial}
  $
  A{_{\mathbf{m}}^{\mathbf{r}}} =h^*(P^{\mathbf{r}}_{\mathbf{m}};x).
  $
\end{lemma}

Lemma \ref{lem: colored eulerian multiset permutation polynomial} provides a generalization of the well-known relationship between the Eulerian polynomials $A_n = A{_{\mathbf{1}}^{\mathbf{1}}}$ and the $h^*$-vector of the $n$-dimensional cube, which is a product of $n$ one-dimensional standard simplices. 
More generally, Lemma~\ref{lem: colored eulerian multiset permutation polynomial} says that the uncolored multiset Eulerian polynomial $A{_{\mathbf{m}}^{\mathbf{1}}}$ is the $h^\ast$-polynomial of an \emph{order polytope} (see \cite[Example~4.6.17]{stanley2011enumerative}), since $P_{\mathbf{m}}^{\mathbf{1}}$ is the order polytope of a naturally labeled collection of disjoint chains. 
We note that a similar polyhedral model, studied in \cite{tielker}, for a different notion of descents in colored multiset permutations agrees with our polyhedral model for $\textbf{m} = \mathbf{1}$ and $\textbf{r} = r\mathbf{1}$ (see \cite[Proposition 5.7]{tielker}). 
Lastly, the case $\textbf{m}=2\mathbf{1}$ and $\textbf{r}=\mathbf{1}$ has also been studied in the context of bipermutohedra \cite{Ardila2020TheB}.

The $m_j$-dimensional simplex $r_j\Delta_{m_j}$ has hyperplane description
\[r_j\Delta_{m_j} = \left\{(x_1, x_2,..., x_{m_j}) \in \mathbb{R}^{{m_j}}: x_1 + x_2 +\cdots+ x_{m_j} \leq r_j \text{ and } x_k \geq 0 \text{ for }k\in [m_j]\right\}. \]
Hence, the direct product $P^{\mathbf{r}}_{\mathbf{m}}$  forms the set
\begin{equation}
\label{eqn:hyperplanes}
\begin{split}
P^{\mathbf{r}}_{\mathbf{m}} & =\{(x_{11}, x_{12},..., x_{1m_1},x_{21}, x_{22},..., x_{nm_n}) \in \mathbb{R}^{{m_1+\cdots+m_n}}: 
\\ & x_{j1} + x_{j2} +\cdots+ x_{jm_j} \leq r_j \text{ and } x_{jk} \geq 0 \text{ for }k\in [m_j] \}.
\end{split}
\end{equation}
Using this observation together with Lemma~\ref{lem: colored eulerian multiset permutation polynomial}, one may deduce the degree of $A{_{\mathbf{m}}^{\mathbf{r}}}$.

\begin{lemma}
    \label{lem:degree}
    $A{_{\mathbf{m}}^{\mathbf{r}}}$ has degree $m + 1 - \max_{k\in[n]}\left\lceil\frac{m_k + 1}{r_k}\right\rceil$.
\end{lemma}

\begin{proof}
It follows by Lemma~\ref{lem: colored eulerian multiset permutation polynomial} and \cite[Theorem 4.5]{BR} that the degree of $A{_{\mathbf{m}}^{\mathbf{r}}}$ is $m+1-s$ where $s$ is the smallest integer such that $sP^{\mathbf{r}}_{\mathbf{m}}$ contains an integer point in its relative interior.

Let $s$ be a positive integer. A point $x\in sP^{\mathbf{r}}_{\mathbf{m}}$ lies in the interior of $s P^{\mathbf{r}}_{\mathbf{m}}$ if and only if all inequalities in the hyperplane description of  $s P^{\mathbf{r}}_{\mathbf{m}}$ are satisfied strictly. Hence, it follows by \eqref{eqn:hyperplanes}
 that $sP^{\mathbf{r}}_{\mathbf{m}}$ contains an integer point in its relative interior if and only if 
all entries in $x$ are strictly positive, and further $x_{j1}+\cdots + x_{jm_j}< sr_j$ for all $j\in [n]$. 
In that case, since all entries of $x$ are at least equal to $1$, the inequality $\frac{m_j+1}{r_j}\leq s$ needs to hold for all $j\in [n]$. 
On the other hand, the vector $\mathbf{1}$ strictly satisfies all defining inequalities that arise from \eqref{eqn:hyperplanes} after dilating by $s$, whenever $\frac{m_j+1}{r_j}\leq s$ for all $j$. We have observed that $sP$ contains an integer point in its relative interior if and only if $\frac{m_j+1}{r_j}\leq s$ for all $j$, from which the claim follows.
\end{proof}

In the following, given a colored multiset Eulerian polynomial $A{_{\mathbf{m}}^{\mathbf{r}}}$, we let $d$ denote its degree (as given in Lemma~\ref{lem:degree}). 
The geometric interpretation of $A{_{\textbf{m}}^{\textbf{r}}}$ given by Lemma~\ref{lem: colored eulerian multiset permutation polynomial} allows us to characterize when $A{_{\textbf{m}}^{\textbf{r}}}$ is symmetric with respect to $d$ using geometric techniques.
We first note some definitions.

The \emph{polar body} of an $n$-dimensional polytope $P$, sometimes called the \emph{dual polytope} of $P$, is the convex body
\[
P^\ast:= \left\{(y_1,\ldots,y_n)\in\R^n : \sum_{i\in[n]}y_ip_i \leq 1, \, \forall (p_1,\ldots,p_n)\in P\right\}.
\]
Note that $P^\ast$ need not be a lattice polytope.  
If both $P$ and $P^\ast$ are lattice polytopes, we say that $P$ is \emph{reflexive}. 
By a result of Stanley \cite{Stanley1978HilbertFO}, $P$ is reflexive if and only if $h^\ast(P; x)$ is symmetric with respect to $n$.
The smallest positive integer $t$ such that $tP$ contains a lattice point in its relative interior is called the \emph{codegree} of $P$, denoted $\codeg(P)$. 
De Negri and Hibi \cite{DH97} proved that if $P\subset \mathbb{R}^n$ is an $n$-dimensional lattice polytope then $h^\ast(P;x)$ is symmetric with respect to its degree if and only if $\codeg(P)P$ contains a unique lattice point $\mathbf{p}$ in its relative interior and $\codeg(P)P - \mathbf{p}$ is reflexive. 
Applying this result, one may obtain the following characterization of the symmetric colored multiset Eulerian polynomials.

\begin{theorem}
    \label{thm: gorenstein}
    $A{_{\mathbf{m}}^{\mathbf{r}}}$ is symmetric (with respect to its degree) if and only if 
    \begin{equation}
        \label{eqn:gorenstein}
        \max_{k\in[n]}\left\lceil\frac{m_k+1}{r_k}\right\rceil r_j  = m_j +1 \qquad \mbox{for all $j\in[n]$.}
    \end{equation}
\end{theorem}

\begin{proof}

By Lemma~\ref{lem: colored eulerian multiset permutation polynomial} it suffices to show that $h^*(P^{\mathbf{r}}_{\mathbf{m}};x)$  is symmetric with respect to $\deg(h^*(P^{\mathbf{r}}_{\mathbf{m}};x))$ if and only if 
$\max_{k\in[n]}\left\lceil\frac{m_k+1}{r_k}\right\rceil r_j  = m_j +1$.
Note that, as can be seen by the proof of Lemma~\ref{lem:degree}, $tP^{\mathbf{r}}_{\mathbf{m}}$ contains an interior lattice point if and only if it contains the lattice point $\mathbf{1}$.
Moreover, it can be seen from the inequalities in~\eqref{eqn:hyperplanes} that this is the unique interior point in $tP^{\mathbf{r}}_{\mathbf{m}}$ if and only if $t = \codeg(P) = \max_{k\in[n]}\left\lceil\frac{m_k+1}{r_k}\right\rceil$ and~\eqref{eqn:gorenstein} is satisfied.
Hence, to examine the symmetry of $h^*(P^{\mathbf{r}}_{\mathbf{m}};x)$ via the result of De Negri and Hibi, we consider the translate $\codeg(P)P^{\mathbf{r}}_{\mathbf{m}}-\mathbf{1}$ of $\codeg(P)P^{\mathbf{r}}_{\mathbf{m}}$ that contains $\mathbf{0}$ in its interior. By \eqref{eqn:hyperplanes}, the polytope $Q^{\mathbf{r}}_{\mathbf{m}}:=\codeg(P)P^{\mathbf{r}}_{\mathbf{m}}-\mathbf{1}$ has hyperplane description
\begin{equation*}
    \begin{split}
    Q^{\mathbf{r}}_{\mathbf{m}} &=\{(x_{11}, x_{12},..., x_{1m_1},x_{21}, x_{22},..., x_{nm_n}) \in \mathbb{R}^{{m_1+\cdots+m_n}}: 
\\ & x_{j1} + x_{j2} +\cdots+ x_{jm_j} \leq \codeg(P)r_j - m_j \text{ and } x_{jk} \geq -1 \text{ for all }k \in [m_j]\}.  
    \end{split}
\end{equation*}
By Lemma~\ref{lem:degree}, $\codeg(P^{\mathbf{r}}_{\mathbf{m}}) = \max_{k\in[n]}\left\lceil\frac{m_k+1}{r_k}\right\rceil$. Hence, if~\eqref{eqn:gorenstein} holds, we see from the above hyperplane description that $Q^{\mathbf{r}}_{\mathbf{m}}$ is reflexive since the constant coefficient in each of $x_{j1} + x_{j2} +\cdots+ x_{jm_j} \leq \codeg(P)r_j - m_j$ is equal to $1$. 
Conversely, if~\eqref{eqn:gorenstein} does not hold, this constant coefficient is not equal to $1$, implying that $Q^{\mathbf{r}}_{\mathbf{m}}$ is not reflexive. 
\end{proof}

\begin{rem}
    \label{rem: s-order symmetry}
    In \cite{Branden2016LectureH}, Br\"and\'en and Leander introduced the \emph{$s$-lecture hall order polytope}, $\mathcal{O}(P,s)$, defined for a sequence of positive integers $s$ and a poset $P$. 
    In \cite[Theorem 4.2]{Branden2016LectureH}, they showed that $h^\ast(\mathcal{O}(P,s); x)$ is symmetric if $P$ is a sign-ranked labeled poset with rank function $\rho$ and $s = \rho + \mathbf{1}$.
    An alternative proof of the ''if'' direction of Theorem~\ref{thm: gorenstein} is given by noting that $A{_{\mathbf{m}}^{\mathbf{r}}} = h^\ast(\mathcal{O}(P,s); x)$ for a sign-ranked poset $P$ with $s = \rho + 1$ and applying their theorem. 
\end{rem}

In \cite{carlitz1978generalized}, Carlitz and Hoggatt showed that $A{_{p\mathbf{1}}^{\mathbf{1}}}$ is symmetric for all $p\geq 1$. 
The same observation also follows from classic results on order polytopes. 
Specifically, the $h^\ast$-polynomial of an order polytope is symmetric if the underlying poset is naturally labeled and graded \cite[Corollary 3.15.18]{stanley2011enumerative}.
As noted above, an uncolored multiset Eulerian polynomial $A{_{\mathbf{m}}^{\mathbf{1}}}$ is the $h^\ast$-polynomial of an order polytope for a poset $P$, which can be taken to be naturally labeled and graded whenever $\mathbf{m} = p\mathbf{1}$. 
These observations are also recovered in the following corollary to Theorem~\ref{thm: gorenstein}.

\begin{corollary}
    \label{cor:symmetricuncolored}
    An (uncolored) multiset Eulerian polynomial $A{_{\mathbf{m}}^{\mathbf{1}}}$ is symmetric if and only if $\mathbf{m} = p\mathbf{1}$ for some $p\geq 1$.
\end{corollary}

\begin{proof}
    Let $k^\ast = \argmax_{k\in[n]}\left\lceil\frac{m_k+1}{r_k}\right\rceil$  {be the value of $k$ that maximizes $\left\lceil\frac{m_k+1}{r_k}\right\rceil$}. 
    Noting that $\mathbf{r} = \mathbf{1}$, Theorem~\ref{thm: gorenstein} implies that $A{_{\mathbf{m}}^{\mathbf{1}}}$ is symmetric if and only if $m_j = m_{k^\ast}$ for all $j\in[n]$. 
\end{proof}

As a second corollary, we also obtain the following characterization of $A{_{\mathbf{m}}^{\mathbf{r}}}$ that are symmetric with respect to degree $m$, which will be used when deducing the self-interlacing property for colored multiset Eulerian polynomials in Subsection~\ref{subsec:self-interlacing}.

\begin{corollary}
    \label{cor: reflexive}
    $A{_{\mathbf{m}}^{\mathbf{r}}}$ is symmetric with respect to degree $m$ if and only if $r_j = m_j + 1$ for all $j\in[n]$.
\end{corollary}

\begin{proof}
    By Lemma~\ref{lem:degree}, $A{{_{\mathbf{m}}^{\mathbf{r}}}}$ is degree $m$ if and only if $\max_{k\in[n]}\left\lceil\frac{m_k + 1}{r_k}\right\rceil = 1$. 
    In this case, the condition of Theorem~\ref{thm: gorenstein} reduces to $r_j = m_j + 1$ for all $j\in[n]$.
\end{proof}
    
As seen from \cite[Theorem 4.5]{BR} and Lemma~\ref{lem: colored eulerian multiset permutation polynomial}, $A{{_{\mathbf{m}}^{\mathbf{r}}}}$ will be degree $m$ if and only if $P^{\mathbf{r}}_{\mathbf{m}}$ contains an interior lattice point. 
It further follows, by \cite[Corollary 10.7]{BR}, that $A{{_{\mathbf{m}}^{\mathbf{r}}}}$ will have a nonnegative $\mathcal{I}_m$-decomposition $(a, b)$ if and only if $P^{\mathbf{r}}_{\mathbf{m}}$ contains an interior lattice point. 
In the next subsection, we will describe when the $a$ and $b$ polynomials in this symmetric decomposition interlace in order to prove the complete catalogue of distributional properties described in Subsection~\ref{sec: preliminaries} for $A{{_{\mathbf{m}}^{\mathbf{r}}}}$. 
To do this, we will describe when $A{{_{\mathbf{m}}^{\mathbf{r}}}}$ is self-interlacing and apply the result of Theorem~\ref{thm:BL}. %\cite[Theorem 2.6]{PS}. 
As nonnegativity of the symmetric decomposition is a necessary condition for applying Theorem~\ref{thm:BL}, we will make specific use of Corollary~\ref{cor: reflexive}. 

\subsection{Self-interlacing Colored Multiset Eulerian Polynomials}
\label{subsec:self-interlacing}
Corollary~\ref{cor: reflexive} can in turn be applied to prove that colored multiset Eulerian polynomials $A{{_{\mathbf{m}}^{\mathbf{r}}}}$ of degree $m$ are self-interlacing, and hence fulfill the complete catalogue of distributional properties described in Subsection~\ref{sec: preliminaries}. 
To do so, one uses the \emph{subdivision operator}, which is the linear transformation on the space of univariate polynomials given by 
% $\mathcal{E}: \binom{x}{k} \longmapsto x^k$. 
\begin{equation*}
    \begin{split}
        \mathcal{E}&: \mathbb{R}[x] \longrightarrow \mathbb{R}[x];\\
        \mathcal{E}&: \binom{x}{k} \longmapsto x^k \qquad \forall k\geq 0,
    \end{split}
\end{equation*}
where $\binom{x}{k} = \frac{x(x-1)\cdots(x-k+1)}{k!}$.
A polynomial $p$ is called $[-1,0]$-rooted if all of its zeros lie in the interval $[-1,0]$.
When applied to a $[-1,0]$-rooted Ehrhart polynomial $\mathcal{L}(P;t)$ that can be expressed as $\mathcal{L}(P;t) = \sum_{i = 0}^dc_it^i(1+t)^{d-i}$ for nonnegative $c_i$, the subdivision operator returns a $[-1,0]$-rooted polynomial \cite[Theorem 4.6]{PB15}. 
This polynomial is the $h^\ast$-polynomial up to a transformation that preserves $(-\infty,0)$-rootedness \cite[Lemma 2.7]{PS}. 
Hence, $h^\ast(P;x)$ is real-rooted.

It can be shown that the Ehrhart polynomial $\mathcal{L}(P^{\mathbf{r}}_{\mathbf{m}}; t)$ is $[-1,0]$-rooted and has nonnegative coefficients in the so-called \emph{magic basis} \cite{Ferroni2023ExamplesAC} 
\[
\{t^i(1+t)^{d-i} : 0\leq i \leq d\}.
\]
Hence, the polynomials $A{_{\textbf{m}}^{\textbf{r}}}$ are all real-rooted. 

In the following, we let $E_k^d := \mathcal{E}(t^k(t+1)^{d-k})$, $0\leq k\leq d$, and we consider the sets
\[
\mathcal{E}_d := \left\{ p = \sum_{k=0}^dc_kE_k^d : c_k\geq0\right\} \qquad \mbox{and} \qquad \mathcal{A}_d := \{p\in \mathcal{E}_d : \mathcal{R}_d(p) \prec p\},
\]
where $\mathcal{R}_d(p) = (-1)^dp(-1 - x)$ as in Theorem~\ref{thm:BL}.
For any degree $d$ Ehrhart polynomial $\mathcal{L}(P; t)$ that is nonnegative in the magic basis, we have that $\mathcal{E}\mathcal{L}(P; t)$ lies in the set $\mathcal{E}_d$. 
Moreover, if $h^\ast(P; x)$ is symmetric with respect to degree $d$ then it has $\mathcal{I}_d$-decomposition $(h^\ast(P; x), 0)$, which is nonnegative and trivially satisfies Theorem~\ref{thm:BL}(1).  
Thus, $\mathcal{E}\mathcal{L}(P; t)$ lies in $\mathcal{A}_d$ by Theorem~\ref{thm:BL}. 
This is true, specifically, for the polynomials $A{_{\textbf{m}}^{\textbf{r}}}$ captured in Corollary~\ref{cor: reflexive}. 
We also have the following observation:
\begin{lemma}
    \cite[Lemma 2.12]{PS}
    \label{lem:PL}
    If $p\in \mathcal{A}_d$, $q\in \mathcal{E}_d$ and $p\prec q$ then $q\in\mathcal{A}_d$. 
\end{lemma}
Combining Lemma~\ref{lem:PL} with Corollary~\ref{cor: reflexive} allows us to deduce the following.
\begin{theorem}
\label{thm: interlacing}
Suppose that $\mathbf{m}$, $\mathbf{r}$ satisfy $r_j\geq m_j+1$ for all $j\in[n]$. Then
 $
\mathcal{I}_{m}A{_{\mathbf{m}}^{\mathbf{r}}}\prec A{_{\mathbf{m}}^{\mathbf{r}}}.
 $
\end{theorem}

\begin{proof}
By Lemma \ref{lem: colored eulerian multiset permutation polynomial} and Corollary~\ref{cor: main id}, $P_{\mathbf{m}}^{\mathbf{r}}$ has Ehrhart polynomial
\[
\mathcal{L}(P_{\mathbf{m}}^{\mathbf{r}};t) = \prod_{j=1}^n\binom{tr_j+m_j}{m_j}.
\]
Since $P_{\mathbf{m}}^{\mathbf{r}}$ is a $m$-dimensional polytope for $m = m_1 + \cdots + m_n$, it follows that $\mathcal{L}(P_{\mathbf{m}}^{\mathbf{r}};t)$ has degree m. 
From this formula it is also clear that $\mathcal{L}(P_{\mathbf{m}}^{\mathbf{r}};t)$ is $[-1,0]$-rooted whenever $r_j \geq m_j + 1$ for all $j$ (or even $r_j \geq m_j$). 

Following the above discussion, we first observe that $\mathcal{L}(P_{\mathbf{m}}^{\mathbf{r}};t)$ has only nonnegative coefficients in the magic basis $\{t^i( 1 + t)^{m - i}: 0 \leq i \leq m\}$. 
To see this, we write $\mathcal{L}(P_{\mathbf{m}}^{\mathbf{r}};t)$ as 
\begin{equation*}
    \begin{split}
        \mathcal{L}(P_{\mathbf{m}}^{\mathbf{r}};t) &= \prod_{j=1}^n\binom{r_jt+m_j}{m_j}
       \\ 
       &=\prod_{j=1}^n\prod_{k = 0}^{m_j - 1}\left(\left(\frac{r_j}{m_j - k}\right)t+1\right),
               \\
       &= \prod_{j=1}^n\prod_{k=0}^{m_j - 1}\left((\frac{r_j}{m_j - k}-1)t+(t+1)\right).
    \end{split}
\end{equation*}
Since $r_j,m_j\geq 1$ and $r_j \geq m_j + 1$ for all $j$, the result follows.
It follows immediately that $\mathcal{E}\mathcal{L}(P_{\mathbf{m}}^{\mathbf{r}};t)\in \mathcal{E}_m$.
Moreover, by the discussion above, \cite[Theorem 4.6]{PB15} implies that $A{_{\textbf{m}}^{\textbf{r}}}$ is real-rooted. 
In particular, by Corollary~\ref{cor: reflexive}, we have that $\mathcal{E}\mathcal{L}(P_{\mathbf{m}}^{\mathbf{r}};t) \in \mathcal{A}_m$ whenever $r_j = m_j + 1$ for all $j$, as in this case $A{_{\textbf{m}}^{\textbf{r}}}$ is symmetric with respect to degree $m$ and Theorem~\ref{thm:BL} applies (note here that $A{_{\textbf{m}}^{\textbf{r}}}$ trivially has a nonnegative $\mathcal{I}_m$-decomposition as it is symmetric with respect to degree $m$).  
Hence, by Lemma~\ref{lem:PL}, to prove the desired result it suffices to show that $\mathcal{E}\mathcal{L}(P_{\mathbf{m}}^{\mathbf{m}+\mathbf{1}};t)\prec \mathcal{E}\mathcal{L}(P_{\mathbf{m}}^{\mathbf{r}};t)$ whenever $r_j \geq m_j +1$ for all $j$.

 Since $\mathcal{L}(P_{\mathbf{m}}^{\mathbf{r}};t)$ has zeros 
$
\beta_{j1} = -\frac{m_j}{r_j}, \beta_{j2} = -\frac{m_j-1}{r_j}, ..., \beta_{jm_j} = -\frac{1}{r_j}$ for $j \in[n]$,
and $\mathcal{L}(P_{\mathbf{m}}^{\mathbf{m}+\mathbf{1}};t)$ has zeros 
$
\alpha_{j1} = -\frac{r_j - 1}{r_j}, \alpha_{j2} = -\frac{(r_j - 1 )-1}{r_j}, ..., \alpha_{jm_j} = -\frac{1}{r_j}$, for $j \in[n]$,
we have that $\alpha_{jk}\leq \beta_{jk}$ for all $j$ and $k$. 
Hence, 
it follows from \cite[Theorem 4.6]{PB15} that $\mathcal{E}\mathcal{L}(P_{\mathbf{m}}^{\mathbf{m}+\mathbf{1}};t) \prec \mathcal{E}\mathcal{L}(P_{\mathbf{m}}^{\mathbf{r}};t)$
whenever $r_j \geq m_j + 1$ for all $j$. 
By Lemma~\ref{lem:PL}, we conclude that $\mathcal{E}\mathcal{L}(P_{\mathbf{m}}^{\mathbf{r}};t)\in\mathcal{A}_m$ whenever $r_j \geq m_j + 1$ for all $j$. 

By Lemma~\ref{lem:degree}, the polytope $P_{\mathbf{m}}^{\mathbf{r}}$ with $h^\ast$-polynomial $A{_{\mathbf{m}}^{\mathbf{r}}}$ contains a lattice point in its relative interior if and only if $r_j \geq m_j + 1$ for all $j$. 
Hence, by \cite[Corollary 10.7]{BR}, $A{_{\mathbf{m}}^{\mathbf{r}}}$ has a nonnegative $\mathcal{I}_m$-decomposition whenever $r_j \geq m_j + 1$ for all $j$. 
Thus, Theorem~\ref{thm:BL} applies, and we conclude that $\mathcal{I}_mA{_{\mathbf{m}}^{\mathbf{r}}} \prec A{_{\mathbf{m}}^{\mathbf{r}}}$ whenever $r_j \geq m_j + 1$ for all $j$.
\end{proof}

The following distributional properties of $A{_\textbf{m}^{\textbf{r}}}$ now follow from Theorem \ref{thm: interlacing} together with Theorem~\ref{thm:BL}(1).
\begin{corollary}
\label{cor: self-interlacing}
If $r_j\geq m_j+1$ for all $j=1,...,n$, then the colored multiset Eulerian polynomial $A{_\textbf{m}^{\textbf{r}}}$ is real-rooted, log-concave, unimodal, alternatingly increasing and bi-$\r$-positive with a real-rooted, log-concave and unimodal $\mathcal{I}_m$-decomposition.
\end{corollary}
This generalizes the observation of Br\"and\'en and Solus \cite[Corollary 3.2]{PS} for colored Eulerian polynomials to the colored multiset Eulerian polynomial setting.

\section{Combinatorial interpretations}
\label{sec:interpretations}

We end with some connections to other lines of combinatorial investigation via alternative interpretations of certain colored multiset Eulerian polynomials and their coefficients.
We pose some general questions that unify several ongoing case-specific studies. 
Using the results in Section~\ref{sec:distributionalproperties}, we complete these specific cases and propose next steps in the context of these more general questions. 
The proposed questions, and our results, pertain specifically to showing when a polynomial is \emph{$s$-Eulerian} \cite{savage2016mathematics} and the search for combinatorial interpretations of the $\gamma$-coefficients in bi-$\gamma$-positive generalizations of Eulerian polynomials.  

\subsection{$s$-Eulerian polynomials}\label{subsec: s-eul}
The $s$-Eulerian polynomials $E_n^s$ are a family of combinatorial generating polynomials defined for any sequence of positive integers $s = (s_1, s_2,\ldots)$. 
They are known to be real-rooted and admit numerous connections to classically studied problems, including the enumeration of lecture hall sequences, as well as connections to discrete geometry where they are interpreted as $h^\ast$-polynomials of \emph{$s$-lecture hall simplices} $P_n^s$; namely, $E_n^s = h^\ast(P_n^s;x)$, where
\[
P_n^s = \{ x\in \mathbb{R}^n : 0\leq x_1/s_1 \leq \cdots \leq x_n / s_n\leq 1\}.
\]
% \danai{we should probably add the s-lecture hall simplex definition.} 
The combinatorics of $s$-lecture hall simplices and  $s$-Eulerian polynomials are surveyed in \cite{savage2016mathematics}. 
It is shown in \cite{savagevisontai}, for instance, that the $r$-colored Eulerian polynomial satisfies $A{_{\mathbf{1}}^{r\mathbf{1}}} = E_n^{(r,2r,\ldots,nr,\ldots)}$, and $A{_{2\mathbf{1}}^{\mathbf{1}}} = E_{2n}^{\hat s}$ for $\hat s = (1,1,3,2,5,3,\ldots)$ (i.e., $\hat s_i = i$ for $i$ odd and $\hat s_i = i/2$ for $i$ even). 
These observations raised the question as to whether or not (colored) multiset Eulerian polynomials are generally \emph{$s$-Eulerian}; i.e., equal to an $s$-Eulerian polynomial for well-chosen $s$. 
\begin{question}
    \label{quest:s-Eulerian}
    When is $A{_{\mathbf{m}}^{\mathbf{r}}}$ $s$-Eulerian?
\end{question}
It was further conjectured by Savage and Visontai in \cite[Conjecture 3.25]{savagevisontai} that $E_{2n}^s = A{_{2\mathbf{1}}^{2\mathbf{1}}}$ when $s = (1,4,3,8,5,\ldots)$ (e.g. $s_{2i} = 4i$ and $s_{2i+1} = 2i + 1$). 
This conjecture was proven in \cite[Section~2]{chen2016s} where the authors additionally showed that $A{_{(2,\ldots, 2,1)}^{2\mathbf{1}}}$ is also $s$-Eulerian. 
To do so, they made the observation that $E_{2n}^{2\hat{s}} = A{_{2\mathbf{1}}^{2\mathbf{1}}}$ and $E_{2n-1}^{2\hat{s}} = A{_{(2,\ldots, 2,1)}^{2\mathbf{1}}}$.
Using Corollary~\ref{cor: main id}, we can generalize this result as follows.
\begin{proposition}
    \label{prop: s-lecture hall}
    The colored multiset Eulerian polynomial $A{_{\mathbf{m}}^{\mathbf{r}}}$ is $s$-Eulerian for any $\mathbf{m}\in\{(1,\ldots,1),(2,\ldots,2),(2,\ldots,2,1)\}$ and $\mathbf{r} = r\mathbf{1}$ for any $r\geq 1$. 
\end{proposition}
%ask for colored bijection between inversion sequences or simplices in s-seqs. and descents of colored multiset permutations for specific multisets

\begin{proof}
In the case $\mathbf{m} = \mathbf{1}$, it is shown in \cite{savagevisontai} that $A{_{\mathbf{1}}^{r\mathbf{1}}} = E_n^{(r,2r,\ldots, nr,...)}$.
For $\mathbf{m} = (2,\ldots,2)$ or $(2,\ldots,2,1)$, it is shown in \cite{SAVAGE2012850} that
\begin{equation*}
    \begin{split}
        \mathcal{L}(P_n^{\hat{s}}; t) &= (t + 1)^{\lfloor\frac{n}{2}\rfloor}\left(\frac{t + 2}{2}\right)^{\lceil\frac{n}{2}\rceil}.
    \end{split}
\end{equation*}
Since $rP_n^{\hat{s}} = P_n^{r\hat{s}}$, we further have that 
\[
\mathcal{L}(P_n^{r\hat{s}}; t) = \mathcal{L}(P_n^{\hat{s}}; rt) = (rt + 1)^{\lfloor\frac{n}{2}\rfloor}\left(\frac{rt + 2}{2}\right)^{\lceil\frac{n}{2}\rceil}. 
\]
When $n$ is even, the right-hand-side is the Ehrhart polynomial for $P_{2\mathbf{1}}^{r\mathbf{1}}$ defined in \eqref{eqn:simprod}, as given by Corollary~\ref{cor: main id} together with Lemma~\ref{lem: colored eulerian multiset permutation polynomial}.
When $n$ is odd, the right-hand-side is the Ehrhart polynomial for $P_{(2,\ldots,2,1)}^{r\mathbf{1}}$.
\end{proof}

It follows from Theorem~\ref{thm: interlacing} that for $r\geq3$ the $s$-Eulerian polynomials captured in Proposition~\ref{prop: s-lecture hall} satisfy the strongest distributional property in Subection~\ref{sec: preliminaries}; namely, they are self-interlacing and hence satisfy the complete catalogue of distributional properties in Subsection~\ref{sec: preliminaries}. 
In the case that $r = 1$, the same is true for $E^{\hat{s}}_{2n}$ since it is known to be symmetric and real-rooted by prior work \cite{savagevisontai} as well as in \cite[Theorem 6.2]{Ardila2020TheB} where these polynomials appear in the study of bipermutohedra. 
$E^{\hat{s}}_{2n-1}$ is also known to have a real-rooted symmetric decomposition \cite{weaklyincreasingtrees}.
Similarly, when $r=2$, it is only known that $E^{2\hat{s}}_{2n}$ and $E^{2\hat{s}}_{2n-1}$ are bi-$\r$-positive.
This connection prompts the following question.
\begin{question}
\label{quest: strongest properties}
    What are the strongest distributional properties satisfied by a given colored multiset Eulerian polynomial?
\end{question}
It is worth noting that even though $E_n^s$ is always real-rooted, it need not have a real-rooted symmetric decomposition. 
For instance, the polynomial in Example~\ref{ex: real-rooted but non-real-rooted dec} is $E_4^s$ for $s = (1,3,1,3,\ldots)$.
An answer to Question~\ref{quest: strongest properties} may give us deeper insights into when an $s$-Eulerian polynomial is self-interlacing, especially given a complete characterization of the $A{_{\mathbf{m}}^{\mathbf{r}}}$ that are $s$-Eulerian (e.g. a complete answer to Question~\ref{quest:s-Eulerian}).
The following observation shows that such a characterization may be challenging to obtain.

\begin{proposition}
    \label{prop:not s-eulerian}
    There exist (un)colored multiset Eulerian polynomials that are not $s$-Eulerian.
\end{proposition}

\begin{proof}
    The proof is by way of example.
    Specifically, we will show that $A{_{\mathbf{m}}^{\mathbf{r}}}$ for $\mathbf{m} = (3,3)$ and $\mathbf{r} = (1,1)$ is not equal to $E_n^s$ for any sequence $s$ and any $n \geq 1$. 
    We first isolate the possible entries of the sequences $s$ that must be considered in closer detail. 
    As shown in  \cite{SAVAGE2012850}, 
    \[
    E_n^s = \sum_{e\in I_n^s}x^{\asc(e)},
    \]
    where 
    \[
    I_n^s = \{e\in \mathbb{Z}^n : 0 \leq e_i \leq s_i -1 \mbox{ for all } i \in [n]\}
    \]
    and 
    \[
    \asc(e) = |\{i \in \{0,\ldots, n-1 : e_i / s_i < e_{i+1} / s_{i + 1}\}|,
    \]
    where $s_0 := 1$ and $e_0 := 0$.
    Hence, $E_n^s(1) = |I_n^s| = \prod_{i = 1}^n s_i$. 
    Since $|\mathfrak{S}{_{\mathbf{m}}^{\mathbf{r}}}| = 20$ for $\mathbf{m} = (3,3)$ and $\mathbf{r} = (1,1)$,
    then if $E_n^s = A{_{\mathbf{m}}^{\mathbf{r}}}$ for some $n\geq 1$ and $s$, it must be that $\prod_{i = 1}^n s_i = 20$. 
    Hence, we consider all factorizations of $20$ into a product of positive integers, and all possible orderings of these factorizations into a sequence $s$.

    According to this observation, the possible $s$ are given by inserting an arbitrary number of $1$s between any of these sequences of orderings of the non-unit factors of $20$. 
    We claim that it suffices to check that $E_n^s \neq A{_{\mathbf{m}}^{\mathbf{r}}}$ whenever there is at most one $1$ between any two of these non-unit factors. 
    Since there are only finitely many such sequences $s$, it can be verified computationally that $E_n^s \neq A{_{\mathbf{m}}^{\mathbf{r}}}$ for these sequences.
    This computation was done using \texttt{Sage} \cite{sagemath}.
    So, to complete the proof, it suffices to prove this claim. 

    Suppose that $s$ is a sequence satisfying $\prod_{i=1}^ns_i = 20$ and that $s_i = 1$ for some $i$. Let $s'$ be the sequence produced by inserting an additional $1$ into position $i+1$ in $s$.
    Note that $|I_n^s| = |I_{n+1}^{s'}|$, which can be see via the bijection $\phi: I_n^s \longrightarrow I_{n+1}^{s'}$ where 
    \[
    \phi: (e_1,\ldots, e_n) \mapsto (e_1,\ldots, e_i, 0, e_{i+1}, \ldots e_n).
    \]
    Specifically, this is a bijection since for any $e\in I_{n+1}^{s'}$, it must be that $e_i = e_{i+1} = 0$, as $s_i = s_{i+1} = 1$. 
    From this observation, it is also clear that $\phi$ preserves the number of ascents; that is, $\asc(e) = \asc(\phi(e))$ for all $e\in I_n^s$. 
    Thus, we conclude that $E_n^s = E_{n+1}^{s'}$, which completes the proof.
\end{proof}

\begin{rem}
    \label{rem:s-eulerian decomps}
Questions \ref{quest:s-Eulerian} and \ref{quest: strongest properties} can be asked more generally for the polynomials in the symmetric decomposition of colored multiset Eulerian polynomials. One observation in that direction is that both polynomials $a,b$ in the $\mathcal{I}_{d}$-decomposition of $A{_{(2,...,2,1)}^{\mathbf{1}}}=E_{2n-1}^{\hat{s}}$ are $s$-Eulerian. In particular, $a=E_{2n+1}^{\tilde{s}}$ where $\tilde{s}=(1,1,3,2,5,3,...)-e_{2n-1}$, and $b=E_{2n}^{\hat{s}}$ where $\hat{s}=(1,1,3,2,5,3,...)$; here $e_{2n-1}$ is the standard basis vector in $\mathbb{R}^{\infty}$. 
This follows by the definitions in the proof Proposition~\ref{prop:not s-eulerian} and the following observation: An ascent always appears in the $2n$-th entry of an inversion sequence $e\in I_{2n+1}^{s}$ whenever $e_{2n+1}=s_{2n+1}-1$. Specifically, $\asc((e_1,...,e_{2n},s_{2n+1}-1))=\asc((e_1,...,e_{2n}))+1$ and the set of inversion sequences in  $I_{2n}^s$ injects into $I_{2n+1}^s$ in this way. 
The remaining inversion sequences in $I_{2n+1}^{s}$ form the set $I_{2n+1}^{\hat{s}}$. 
Comparing the degrees of the corresponding $s$-Eulerian polynomials and using the fact that $I_{2n}^s=A{_{2\mathbf{1}}^{\mathbf{1}}}$ is symmetric with respect to its degree by Theorem \ref{thm: gorenstein} and the uniqueness of the symmetric decomposition completes the argument.
\end{rem}

\begin{rem}
    \label{rmk:orderpolytopes}
    In \cite{Branden2016LectureH}, Br\"and\'en and Leander introduced a common generalization of $s$-lecture hall simplices and order polytopes called \emph{$s$-lecture hall order polytopes} (see Remark~\ref{rem: s-order symmetry}). 
    If $p = h^\ast(\mathcal{O}(P, s); x)$, we say that $p$ is \emph{$(P, s)$-Eulerian}. 
    Question~\ref{quest:s-Eulerian} asks when $A{_{\mathbf{m}}^{\mathbf{r}}}$ is $s$-Eulerian. Proposition~\ref{prop:not s-eulerian} shows that not all colored multiset Eulerian polynomials have this property. On the other hand, $A{_{\mathbf{m}}^{\mathbf{r}}}$ is always $(P,s)$-Eulerian. 
\end{rem}

\subsection{Interpreting $\gamma$-coefficients}

Corollary~\ref{cor: self-interlacing} yields a large family of relatively unexplored combinatorial generating polynomials that are bi-$\gamma$-positive.  
A natural endeavour is to explore combinatorial interpretations of these $\gamma$-coefficients.  
Such investigations may prove fruitful in, for instance, the development of a complete understanding of the $\gamma$-coefficients for $s$-Eulerian polynomials when interpreted through the lens of Proposition~\ref{prop: s-lecture hall}.
Hence, in this section, we provide results pertaining to the following general question:
\begin{question}
    \label{prob: gamma coeffs}
    Let $(a,b)$ be the $\mathcal{I}_d$-decomposition of a bi-$\gamma$-positive colored multiset Eulerian polynomial $A{_{\mathbf{m}}^{\mathbf{r}}}$, where $d$ is the degree of $A{_{\mathbf{m}}^{\mathbf{r}}}$.  
    What is a combinatorial interpretation of the coefficients of $a$ and $b$ when expressed in the $\gamma$-basis?
\end{question}

Corollary~\ref{cor: self-interlacing} provides a general family of colored multiset Eulerian polynomials that are bi-$\gamma$-positive for which Question~\ref{prob: gamma coeffs} may be considered.
This family includes the previously studied colored Eulerian polynomials $A{_{\mathbf{1}}^{r\mathbf{1}}}$ for $r\geq 1$.
In \cite[Question 2.27]{athanasiadis2017gamma}, Athanasiadis considers a special case, asking for a solution to Question~\ref{prob: gamma coeffs} for $ A{_{\mathbf{1}}^{r\mathbf{1}}}$ for $r\geq 1$. 
A solution to Athanasiadis's question would extend the well-known interpretation of the $\gamma$-coefficients for the Eulerian polynomial $A_n = A{_{\mathbf{1}}^{\mathbf{1}}}$ due to Foata and Strehl \cite{foata1973euler} (see for instance \cite[Section~3.1]{PB}), to all colored Eulerian polynomials. 
Hence, one could more generally solve Athanasiadis's question by solving Problem~\ref{prob: gamma coeffs} for all $A{_{\mathbf{m}}^{\mathbf{r}}}$ captured in Corollary~\ref{cor: self-interlacing}.

Previous works have also considered Problem~\ref{prob: gamma coeffs} in the case when $\mathbf{m} \neq \mathbf{1}$ and $\mathbf{r} = \mathbf{1}$.
In this case, we see by Lemma~\ref{lem:degree} that $d = \deg(A{_{\mathbf{m}}^{\mathbf{r}}})< m$. 
Thus, we are not in the situation of Corollary~\ref{cor: self-interlacing}, nor do we even know that the $\mathcal{I}_d$-decomposition of $A{_{\mathbf{m}}^{\mathbf{r}}}$ is nonnegative, let alone $\gamma$-positive. 
Consequently, only partial results in this case are currently known, which (similar to the previous results on Question~\ref{quest:s-Eulerian}) focus on the case when $\mathbf{m} \in\{ k\mathbf{1}, (k,\ldots, k,1)\}$. 

Specifically, by Corollary~\ref{cor:symmetricuncolored}, $A{_{\mathbf{m}}^{\mathbf{1}}}$ is symmetric with respect to its degree if and only if $\mathbf{m} = k\mathbf{1}$ for some $k\geq 1$. 
Since $A{_{k\mathbf{1}}^{\mathbf{1}}}$ has only nonnegative coefficients and is real-rooted, it is $\gamma$-positive with $\mathcal{I}_d$-decomposition $(A{_{k\mathbf{1}}^{\mathbf{1}}}, 0)$.
Lin, Ma, Ma and Zhou used the notion of weakly increasing trees to provide a combinatorial interpretation of the $\gamma$-coefficients of $A{_{2\mathbf{1}}^{\mathbf{1}}}$ \cite{weaklyincreasingtrees}, which was subsequently extended to a combinatorial interpretation of the $\gamma$-coefficients of $A{_{k\mathbf{1}}^{\mathbf{1}}}$ for all $k\geq 1$ in \cite[Theorem 1.3]{multiseteulerian}. 
In particular, the $i$-th $\gamma$-coefficient of $A{_{k\mathbf{1}}^{\mathbf{1}}}$ counts the number of weakly increasing trees on $\{1^k,2^k,\cdots,(n-1)^k,n\}$ with $i+1$ leaves and no young leaves (as defined below in Subsubsection~\ref{subsubsec:gammaproof}).  
Hence, by Corollary~\ref{cor:symmetricuncolored}, these results provide a complete answer to Problem~\ref{prob: gamma coeffs} when $A{_{\mathbf{m}}^{\mathbf{1}}}$ is symmetric.

 Perhaps the simplest case in which $A{_{\mathbf{m}}^{\mathbf{1}}}$ is not symmetric is when $\mathbf{m} = (k,\ldots, k, 1)$ for $k> 1$. 
 In this case, Ma, Ma and Yeh \cite{ma2019eulerian} showed that the polynomial $A{_{(k,\ldots,k,1)}^{\mathbf{1}}}$ is bi-$\gamma$-positive when $k=2$.  
 Lin, Xu and Zhao extended the bi-$\gamma$-positivity of $A{_{(k,\ldots,k,1)}^{\mathbf{1}}}$ to all $k$ and showed that the $b$-polynomial in its symmetric decomposition is a scalar multiple of $A{_{k\mathbf{1}}^{\mathbf{1}}}$ \cite[Theorem 1.5]{multiseteulerian}, thereby yielding an interpretation of the $\gamma$-coefficients of the $b$-polynomial in the $\mathcal{I}_d$-decomposition $(a,b)$ of $A{_{(k,\ldots,k,1)}^{\mathbf{1}}}$ for all $k> 1$.
 The question of interpreting the $\gamma$-coefficients of the $a$-polynomial was left open. 
 In Subsection~\ref{subsubsec:gammaproof}, we complete their partial solution for this case of Problem~\ref{prob: gamma coeffs} by proving the following.
 
 \begin{theorem}
 \label{thm: gamma-interpretation}
Let $(a,b)$ be the $\mathcal{I}_d$-decomposition of $A{_{(k,\ldots,k,1)}^{\mathbf{1}}}$, where $d$ denotes its degree. Then
\begin{equation} \label{m:gam}
a=\sum\limits_{i=0}^{(n-1)}\r_{n,i}x^i(1+x)^{2(n-1)+1-2i},
\end{equation}
where $\r_{n,i}$ is the number of  weakly increasing trees on $\{1^k,2^k,\cdots,(n-2)^k,n-1,n\}$ with $i+1$ leaves and no young leaves.
\end{theorem}

\begin{rem}
    \label{rmk:gammacoeffs}
    Combining Theorem~\ref{thm: gamma-interpretation} with the results of Lin, Xu and Zhao \cite[Theorems 1.3 and 1.5]{multiseteulerian}, and Foata and Strehl \cite{foata1973euler} gives a complete interpretation of the $\gamma$-coefficients of the $\mathcal{I}_d$-decomposition of the multiset Eulerian polynomial $A{_{\mathbf{m}}^{\mathbf{1}}}$ for $\mathbf{m} \in\{k\mathbf{1}, (k,\ldots, k, 1)\}$. 
    Extending these results on Problem~\ref{prob: gamma coeffs} to the case in which $\mathbf{m} \in\{k\mathbf{1}, (k,\ldots, k, 1)\}$ and $\mathbf{r} = r\mathbf{1}$ for $r\geq 1$, would both provide an answer to the question of Athanasiadis \cite[Question 2.27]{athanasiadis2017gamma} and provide insights into the $\gamma$-coefficients of bi-$\gamma$-positive $s$-Eulerian polynomials via Proposition~\ref{prop: s-lecture hall}.
\end{rem}
{Theorem \ref{thm: gamma-interpretation} provides a combinatorial interpretation of the $\gamma$-coefficients arising in the symmetric decomposition of the $s$-Eulerian polynomial $E_{2n-1}^{\hat{s}} = A{_{(2,\ldots, 2,1)}^{\mathbf{1}}}$, discussed in Section \ref{subsec: s-eul}. Extending this understanding to other $s$-Eulerian polynomials outside the class of multiset Eulerian polynomials is another natural direction for future work. One such (non-overlaping) class is recently investigated by Yan, Yang and Lin in \cite{YAN2026106092}.}
{
{We also note that a different family of bi-$\gamma$-positive generalizations of Eulerian polynomials whose coefficients have a combinatorial interpretation are considered in \cite{MR4682282}.}}

\subsubsection{Proof of Theorem \ref{thm: gamma-interpretation}}
\label{subsubsec:gammaproof}
We start with a few definitions. A {\em plane tree} is a rooted tree in which the children of each node are ordered. 
\begin{definition}[Weakly increasing trees~\cite{weaklyincreasingtrees}]\label{def: weaktrees}
Let $M=\{1^{m_1},2^{m_2},\cdots,n^{m_n}\}$ be a multiset.
A {\em weakly increasing tree} on $M$ is a plane tree  that: 
\begin{enumerate}[label=(\roman*)]
\item  contains $|M|+1$ nodes, labeled by the elements in $M\cup\{0\}$,
\item  its root is labeled by $0$ and each node receives a  label weakly greater than its parent, 
\item  the labels of the children of each node are weakly increasing from left to right. 
\end{enumerate}   
Let $\mathcal{T}_{M}$ be the set of weakly increasing trees on the multiset $M$. 
\end{definition}
%A leaf or an internal node (other than the root) of a weakly increasing tree is {\em old} if it is the rightmost child of its parent; otherwise, it is {\em young} \cite{youngleaves}.

The proof of Theorem \ref{thm: gamma-interpretation} follows from the methods in \cite{multiseteulerian}.
In particular, it makes use of the \emph{$M$-Eulerian-Narayana polynomial} introduced in \cite{weaklyincreasingtrees}. 
% Recall from Definition \ref{def: weaktrees} that $\mathcal{T}_M$ denotes the set of weakly increasing trees on a multiset $M$.

\begin{definition}
    Let $t\neq 0$. The node $t$ of a weakly increasing tree ${T}$ is an {\em internal node} if $t$ has at least one child; otherwise $t$ is called a {\em leaf}. A leaf or an internal node of a weakly increasing tree is {\em old} if it is the rightmost child of its parent; otherwise, it is {\em young} \cite{youngleaves}.

    We denote the number of internal nodes in $T$ by int$({T})$.
    For a multiset $M$, the \emph{$M$-Eulerian–Narayana polynomial} is 
    \[
    A_{\mathcal{T}_M}=\sum\limits_{T\in \mathcal{T}_M}x^{\text{int}(T)}.
    \]
\end{definition}

The $\{1^k,2^k,...,(n-1)^k,n\}$-Eulerian-Narayana polynomial coincides with the multiset Eulerian polynomial $A{_{k\mathbf{1}}^\mathbf{1}}$ for all $k$.
In particular, it was proven in \cite[Theorem 3.3]{multiseteulerian} that
\begin{equation}\label{eq: weaktrees and multisets} A{_{k\mathbf{1}}^\mathbf{1}}=\sum\limits_{T\in \mathcal{T}_{\{1^k,2^k,...,(n-1)^k,n\}}}x^{\text{int}(T)}.\end{equation}

Despite the evident connection between the two polynomials, no bijection that maps ascents of multipermutations to internal nodes of weakly increasing trees has been established. 
The proof of \eqref{eq: weaktrees and multisets} relied on so-called {\em multiset trees}. 
These are precisely the weakly increasing trees where each child of a node has a label strictly greater than its parent.
The notation $\mathcal{P}_M$ is used to denote the set of multiset trees lying in $\mathcal{T}_M$ for a given multiset $M$. 

Lin, Ma, Ma and Zhou showed that there is a bijection between the multiset trees of a multiset $M$ and its multiset permutations, such that internal nodes are mapped to ascents (which are seen to be equidistributed with descents) \cite{weaklyincreasingtrees}.
The bijection goes as follows: 
For $\pi \in \mathfrak{S}{_{\mathbf{m}}}$, let $T_\pi$ be the tree on vertices $M\cup \{0\}$ in which for each $i$ in $\pi$, we take $i$ to be the child of the rightmost element $j$ preceding $i$ in $\pi$ for which $j < i$. 
If no such $j$ exists, then $i$ is a child of $0$. 
$T_\pi$ is constructed iteratively, such that each new child added to a node is placed furthest to the left (e.g., placed as the youngest child). 
For example, for $\pi=324313244312\in  \mathfrak{S}{_{\mathbf{m}}}$, $T_\pi$ is the multiset tree in Figure~\ref{fig:example}.

\begin{figure}%[h]
\centering
\begin{tikzpicture}[scale=0.3]
%画点
\node at (20,20) {$\bullet$};
\node at (15,17) {$\bullet$};
\node at (18.5,17) {$\bullet$};
\node at (21.5,17){$\bullet$};
\node at (21,14){$\bullet$};
\node at (23,14){$\bullet$};

\node at (25,17) {$\bullet$};
\node at (17,14) {$\bullet$};
\node at (16,11) {$\bullet$};
\node at (17,11) {$\bullet$};
\node at (18,11) {$\bullet$};

\node at (19,14) {$\bullet$};
\node at (14,14){$\bullet$};

\node at (20,21) {$0$};
\node at (14.3,17) {$1$};
\node at (17.8,17) {$1$};
\node at (22.2,17){$2$};
\node at (25.6,17) {$3$};
\node at (14,13) {$2$};
\node at (16.2,14) {$2$};
\node at (19,13) {$3$};
\node at (21,13) {$3$};
\node at (23,13) {$4$};
\node at (16,10) {$3$};
\node at (17,10) {$4$};
\node at (18,10) {$4$};

\draw[-] (20,20) to (15,17);
\draw[-] (20,20) to (18.5,17);
\draw[-] (20,20) to (21.5,17);
\draw[-] (20,20) to (25,17);
\draw[-] (15,17) to (14,14);
\draw[-] (18.5,17) to (17,14);
\draw[-] (18.5,17) to (19,14);
\draw[-] (17,14) to (16,11);
\draw[-] (17,14) to (17,11);
\draw[-] (17,14) to (18,11);
\draw[-] (21.5,17) to (21,14);
\draw[-] (21.5,17) to (23,14);
\end{tikzpicture}
\caption{A multiset tree on $M=\{1^2,2^3,3^4,4^3\}$.\label{fig:example}}
\end{figure}
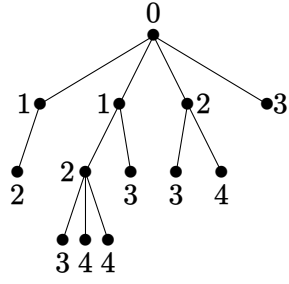

Let us now turn our attention to the multiset $M=\{1^k,...,(n-1)^k,n\}$ and the symmetric decomposition of $A{_{(k,...,k,1)}^{\mathbf{1}}}$. 
The permutations of $M$ can be split into two sets:
\begin{itemize}
    \item[$a)$] those arising from a permutation of $\{1^k,...,(n-1)^k\}$ by adding $n$ to the right of any but the rightmost copy of $n-1$ %($k-1$ candidates in total) 
    and 
    \item[$b)$] those arising from inserting $n$ in any other location. %those arising by adding $n$ to the right of any copy of some $1\leq i \leq n-2$ or the rightmost copy of $n-1$, or in the beginning.
    \end{itemize}
    Equivalently, we can consider: 
    \begin{itemize}
        \item[$a)$] the set $G_n^{(k)}$ of multiset trees on $M$ where $n$ is the child of $n-1$ which appears in the first $k-1$ appearances of $n-1$ from right to left, and 
        \item[$b)$] the set $H_n^{(k)}$ of multiset trees on $M$ where $n$ is not the child of  $n-1$ which appears in the first $k-1$ appearances of $n-1$ from right to left.
\end{itemize}

\begin{lemma}\label{lem: decomposition trees} For the symmetric decomposion $a_n+xb_n$ of $A{_{(k,...,k,1)}^{\mathbf{1}}}$, we have:
\begin{align}
xb_n=\sum_{\pi\in G_n^{(k)}} x^{\intt(T)},\label{tree:tb}\\
a_n=\sum_{\pi\in H_n^{(k)}} x^{\intt(T)}.\label{tree:a}
%A_n^{(p)}(t)=\sum_{T\in\P_n^{(p)}} t^{\intt(T)-1}. 
\end{align}
\end{lemma}

\begin{proof}
The number of permutations of $M$ with the copy of $n$ placed next to a specific copy of $n-1$ is fixed and equal to the number of multiset permutations of $\{1^k,...,(n-1)^k\}$.
Further, placing a copy of $n$ next to a copy of $(n-1)$ increases the number of ascents by one. (Recall that ascents and descents are equidistributed for uncolored multiset permutations.)
Considering the restriction to use any of the first $k-1$ copies of $n-1$, and the bijection between multiset permutations and multiset trees, it follows that  {$\sum\limits_{\pi\in G_n^{(k)}} x^{\intt(T)}=x(k-1)A_{\mathbf{m}}$, where the multiset here is $M_{\mathbf{m}}=\{1^k,...,(n-1)^k\}$}.
In \cite[Theorem 1.5]{multiseteulerian}, it was shown that  {$b_n=(k-1)A_{\mathbf{m}}$, where again $M_{\mathbf{m}}=\{1^k,...,(n-1)^k\}$}. Therefore, the equality in \eqref{tree:tb} holds by uniqueness of the symmetric decomposition.
This completes the proof, since the sets $G_n^{(k)}$ and $H_n^{(k)}$ partition the space of multiset trees on $M$.
\end{proof}

%We are now ready to prove Theorem \ref{thm: gamma-interpretation}.

    \begin{lemma}\label{lem: bijection}
    \begin{equation}\label{eq: towards gamma interpretation}\sum\limits_{T\in H_n^{(k)}}x^{\text{int}(T)}=\sum\limits_{T\in \mathcal{T}_{\{1^k,...,(n-2)^k,n-1,n\}}}x^{\text{int}(T)}.\end{equation}
    \end{lemma}

\begin{proof}
    The proof makes use of the identity in equation \eqref{eq: weaktrees and multisets}. For simplicity, we will work with the following statement, which is equivalent to equation \eqref{eq: weaktrees and multisets} (\cite[Theorem 3.3]{multiseteulerian}):
    \begin{equation}\label{eq: theorem 3.3 in multiset paper}\sum\limits_{T\in \mathcal{P}_{\{1^k,...,(n-1)^k\}}}x^{\text{int}(T)}=\sum\limits_{T\in \mathcal{T}_{\{1^k,...,(n-2)^k,n-1\}}}x^{\text{int}(T)}.\end{equation}

Let us first compare the cardinalities of the sets $H_n^{(k)}$ and $\mathcal{T}_{\{1^k,...,(n-2)^k,n-1,n\}}$.
The permutations in $H_n^{(k)}$ can be generated by inserting the copy of $n$ in every tree in $\mathcal{P}_{\{1^k,...,(n-1)^k\}}$. 
Since there are $(n-1)k+1$ nodes in $\mathcal{P}_{\{1^k,...,(n-1)^k\}}$ in total, and $k-1$ of those are forbidden, this leaves $(n-2)k+2$ nodes to receive $n$ as a child, i.e., \[|H_n^{(k)}|=((n-2)k+2)|\mathcal{P}_{\{1^k,...,(n-1)^k\}}|.\] 
Similarly, a weakly increasing tree on $\{1^k,...,(n-2)^k,n-1,n\}$ can be formed by inserting $n$ in every tree in $\mathcal{T}_{\{1^k,...,(n-2)^k,n-1\}}$. 
There are $(n-2)k+2$ nodes in $\mathcal{T}_{\{1^k,...,(n-2)^k,n-1\}}$ and each of them is eligible to receive $n$ as a child. 
Hence it follows that
\[|\mathcal{T}_{\{1^k,...,(n-2)^k,n-1,n\}}|=((n-2)k+2)|\mathcal{T}_{\{1^k,...,(n-2)^k,n-1\}}|.\]
Hence, $|H_n^{(k)}|=|\mathcal{T}_{\{1^k,...,(n-2)^k,n-1,n\}}|$.

To prove \eqref{eq: towards gamma interpretation}, we will show that there is an injection $\psi: H_n^{(k)} \longrightarrow \mathcal{T}_{\{1^k,...,(n-2)^k,n-1,n\}}$ that preserves the number of internal nodes.
Given a tree $T_n$ in $ H_n^{(k)}$, there exists a tree $T_{n-1}$ in $\mathcal{P}_{\{1^k,...,(n-1)^k\}}$ such that $T_n$ is constructed from $T_{n-1}$ by inserting $n$ as a child of the $i$-th internal node (resp. leaf) of $T_{n-1}$, where the % internal nodes (resp. leaves) are enumerated from left to right and from the root towards the leaves 
 {leaves are enumerated from left to right and the internal nodes are also enumerated from left to right, with priority given to the nodes closer to the root for nodes that lie in the same chain. }
Then, by \eqref{eq: theorem 3.3 in multiset paper}, there exists a weakly increasing multiset tree $\phi(T_{n-1})\in \mathcal{T}_{\{1^k,...,(n-2)^k,n-1\}}$ that has the same number of internal nodes as $T_{n-1}$.
We proceed by inserting $n$ as the $i$-th internal node (resp. leaf) of $\phi(T_{n-1})$, again counting from left to right and from root to leaves  {(if $n$ was inserted as a child of the root of $T_{n-1}$, we insert $n$ as a child of the root of $\phi(T_{n-1})$)}.
We choose this as $\psi(T_n)$. 
 {
For example, in Figure \ref{fig: bijection gamma pos}, there is exactly one internal node in each tree (which in both cases it has label 1). There are six leaves on the left tree that lies in $\T_{\{1^3,2^3,3\}}$, all eligible to receive $n=4$ as a child, and they are ordered from left to right, i.e., as $1,1,2,2,2,3$. On the other hand, the multiset tree on the right that lies in $\P_{\{1^3,2^3,3^3\}}$ has eight leaves but two of them are forbidden from receiving $4$ as a child (red nodes). The remaining six leaves are ordered from left to right, i.e., as $1,2,3,1,2,2$. Then $4$ is inserted to the third leaf in each tree. In other words, if $\phi(T_{n-1})=\T_{\{1^3,2^3,3\}}$, then the multiset tree on the right would be mapped to the weakly increasing tree on the left through $\phi$.} 

This provides the desired injection.
\end{proof}

\begin{figure}%[h]
\centering
\begin{tikzpicture}[scale=0.2]
%%画点

\node at (0,20) {$\bullet$};
\node at (-4,16) {$\bullet$};
\node at (4,16) {$\bullet$};
\node at (0,16) {$\bullet$};
\node at (0,16) {$\bullet$};
\node at (-4.5,12) {$\bullet$};
\node at (-1.5,12) {$\bullet$};
\node at (1.5,12) {$\bullet$};
\node at (4.5,12) {$\bullet$};
\node at (20,20) {$\bullet$};
\node at (20,15) {$\bullet$};
\node at (15.5,16.5) {$\bullet$};
\node at (17.5,15) {$\bullet$};
\node at (22.5,15) {$\bullet$};
\node at (24.5,16.5) {$\bullet$};
\node at (13,11)  {$\bullet$};
\node at (16,11)  {$\bullet$};
\node at (19,11)  {$\color{red}{\bullet}$};
\node at (22,11)  {$\color{red}{\bullet}$};
%%%%%%%%%%%%%%%%%%%
\node at (-1.5, 8){$\bullet$};
\node at (16, 7){$\bullet$};
%%数字
\node at (0,21.5) {$0$};
\node at (-5,16.5) {$1$};
\node at (-1,16.5) {$1$};
\node at (3,15.5) {$3$};
\node at (-5,11) {$1$};
\node at (-2.5,11) {$2$};
\node at (1,11) {$2$};
\node at (4,11) {$2$};
\node at (20,21.5) {$0$};
\node at (14.5,17) {$1$};
\node at (16.8,16) {$1$};
\node at (19.2,15.5) {$1$};
\node at (23.4,15.3) {$2$};
\node at (25.5,17) {$2$};
\node at (12.5,10) {$2$};
\node at (15.0,10) {$3$};
\node at (18.5,10) {$3$};
\node at (21.5,10) {$3$};
%%%%%%%%%%%%%%%%%%%%
\node at (-2.5, 7){$4$};
\node at (15.0, 6){$4$};

%%画边
\draw[-] (0,20) to (0,16);
\draw[-] (0,20) to (-4,16);
\draw[-] (0,20) to (4,16);
\draw[-] (0,16) to (-4.5,12);
\draw[-] (0,16) to (-1.5,12);
\draw[-] (0,16) to (1.5,12);
\draw[-] (0,16) to (4.5,12);
\draw[-] (20,20) to (20,15);
\draw[-] (20,20) to (15.5,16.5);
\draw[-] (20,20) to (17.5,15);
\draw[-] (20,20) to (22.5,15);
\draw[-] (20,20) to (24.5,16.5);
\draw[-] (17.5,15) to (13,11);
\draw[-] (17.5,15) to (16,11);
\draw[-] (17.5,15) to (19,11);
\draw[-] (17.5,15) to (22,11);
%%%%%%%%%%%%%%%%%%%
\draw[dashed] (-1.5, 12) to (-1.5, 8);
\draw[dashed](16, 11) to (16, 7);
\end{tikzpicture}
\caption{ Inserting $4$  {into} a weakly increasing tree in $\T_{\{1^3,2^3,3\}}$ (left) and a multiset tree in $\P_{\{1^3,2^3,3^3\}}$ (right). The red nodes are forbidden for receiving $4$ as a child. \label{fig: bijection gamma pos}}
\end{figure}
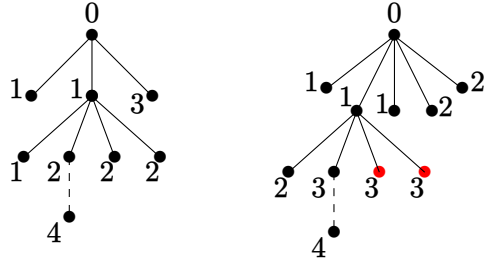

\begin{proof}[Proof of Theorem~\ref{thm: gamma-interpretation}]
The right-hand-side in \eqref{eq: towards gamma interpretation} is an $M$-Eulerian-Narayana polynomial. The $M$-Eulerian-Narayana polynomials are $\gamma$-positive, and the $i$-th $\gamma$-coefficient counts the number of weakly increasing trees on $M$ with $i + 1$ leaves and without young leaves \cite[Theorem 2.12]{weaklyincreasingtrees}. 
Combining this interpretation with Lemmas~\ref{lem: decomposition trees} and~\ref{lem: bijection}  {complete} the proof of Theorem \ref{thm: gamma-interpretation}.
\end{proof}

 \section*{Acknowledgements} \label{sec:acknowledgements}
 We thank Matthias Beck and Petter Br\"and\'en for helpful discussions,  {as well as the anonymous referees for useful suggestions}.
 Part of this research was performed while the authors were visiting the Institute for Mathematical and Statistical Innovation (IMSI), which is supported by NSF Grant No.~DMS-1929348.
 Danai Deligeorgaki and Liam Solus were partially supported by the Wallenberg Autonomous Systems and Software Program (WASP) funded by the Knut and Alice Wallenberg Foundation. 
 Liam Solus was further supported by a Research Pairs grant from the Digital Futures Lab at KTH, the Göran Gustafsson Foundation Prize for Young Researchers, and Starting Grant No. 2019-05195 from The Swedish Research Council.

\bibliographystyle{abbrvnat}
\bibliography{newbib}
\end{document}